\documentclass[12pt]{article}
\usepackage{graphicx}
\usepackage{amsmath}
\usepackage{amssymb}
\usepackage{theorem}
\usepackage{enumerate} 
\usepackage{color}

\usepackage{hyperref}

\numberwithin{equation}{section}

\newtheorem{thm}{Theorem}[section]
\newtheorem{lemma}[thm]{Lemma}

\newtheorem{prop}[thm]{Proposition}
\newtheorem{cor}[thm]{Corollary}
{\theorembodyfont{\rmfamily}

\newtheorem{fact}[thm]{Fact}

\newtheorem{rmk}[thm]{Remark}
}

\newcommand{\qed}{\hfill \mbox{\raggedright \rule{.07in}{.1in}}}
 
\newenvironment{proof}{\vspace{1ex}\noindent{\bf
Proof}\hspace{0.5em}}{\hfill\qed\vspace{1ex}}
\newenvironment{pfof}[1]{\vspace{1ex}\noindent{\bf Proof of
#1.}\hspace{0.5em}}{\hfill\qed\vspace{1ex}}

\newcommand{\cP}{{\mathbb P}}
\newcommand{\R}{{\mathbb R}}

\newcommand{\Z}{{\mathbb Z}}

\newcommand{\N}{{\mathbb N}}
\newcommand{\E}{{\mathbb E}}
\newcommand{\cB}{{\mathcal B}}

\newcommand{\supI}{{\SMALL \sup_I}}

\newcommand{\Var}{\operatorname{Var}}

\newcommand{\BV}{{\rm BV}}

\newcommand{\eps}{{\epsilon}}

\newcommand{\SMALL}{\textstyle}

\newcommand{\vertiii}[1]{{\left\vert\kern-0.25ex\left\vert\kern-0.25ex\left\vert #1
    \right\vert\kern-0.25ex\right\vert\kern-0.25ex\right\vert}}

\title{Stable large deviations for deterministic dynamical systems}
\author{
Jonny  Imbierski
\and
Dalia Terhesiu
\thanks{The affiliation of both authors is
Mathematisch Instituut,
University of Leiden, Niels Bohrweg 1, 2333 CA Leiden, Netherlands.
Email addresses: j.f.imbierski@math.leidenuniv.nl and daliaterhesiu@gmail.com}
}

\begin{document}

 \maketitle
 
 \begin{abstract}
We obtain large deviations for a class of non-square-integrable dependent random variables in the domain of attraction of an $\alpha$-stable law, $\alpha\in(0,1)\cup(1,2]$. 
This class includes ergodic sums of observables in the domain of attraction of an $\alpha$-stable law driven by Gibbs-Markov maps.
\end{abstract}

\section{Introduction}

The purpose of this work is to obtain large deviations (LD) for a  class of dependent random variables satisfying $\alpha$-stable limit laws, $\alpha\in(0,1)\cup(1,2]$. In particular, we obtain stable large deviations for a class of
 deterministic dynamical systems (including Gibbs-Markov maps, see Section~\ref{sec-GM} for details).
 As far as we are aware this is the first stable LD result for dynamical systems. 

We start by  recalling the statement of LD for real-valued i.i.d.\ random variables $(X_i)_{i\ge 1}$ characterised by stable laws with index $\alpha\in (0,2]$; for the $\alpha=2$ case we still assume 
that the variance is infinite. 
Suppose that
\begin{equation}\label{eq-rvbt}
\cP(X_1>x)=p\ell(x) x^{-\alpha}(1+o(1)),\quad
\cP(X_1<-x)=q\ell(x) x^{-\alpha}(1+o(1)),
\end{equation} 
as $x \to \infty$, where $\ell$ is a slowly varying function and $p,q>0$ with $p+q=1$.
Then the sequence $(X_i)_{i\ge 1}$ satisfies the following convergence result in distribution:
\begin{equation}
\label{eq:stlaw}
 (S_n-b_n)/a_n\to_d Y_\alpha, 
\end{equation}
where $S_n=\sum_{i=1}^n X_i$, the limit random variable $Y_\alpha$ is an $\alpha$-stable law (with parameters $\alpha,p, q$),
\begin{align}\label{eq:hatell}
a_n^{\alpha}(1+o(1)
)=\begin{cases}
 n\ell(a_n),&\text{ if }\alpha\in (0,2)\\
 n\hat\ell(a_n),&\text{ if }\alpha=2,\text{ for } 
 \hat\ell(x) := 1+\int_1^{1+x}\frac{\ell(u)}{u}\, du,
\end{cases}
\end{align} 
and 
\begin{equation}
\label{eq:bn}
b_n=\begin{cases}
      0,&\text{ if }\alpha\in (0,1),\\
      n\E(X_1),&\text{ if }\alpha\in (1,2],\\
      n\E( X_1 1_{\{ |X_1|\le a_n\} }),&\text{ if }\alpha=1
     \end{cases}
\end{equation} is the centring sequence.
It is known that in the i.i.d.\ set-up,~\eqref{eq-rvbt} and~\eqref{eq:stlaw}
are equivalent.

We recall that when $\alpha = 2$,  up to redefining $X_1$ as $X_1- \E(X_1)$ and re-scaling $a_n$ by a constant, the limit random variable $Y_2$
is the standard Gaussian.
Let $\bar\Phi(x) = \frac{1}{x \sqrt{2\pi}} e^{-x^2/2} (1+o(1))$ be the distribution function of the standard Gaussian. 
Let $\sigma^2(x) = \E( (X_1-\E(X_1))^2 1_{\{|X_1-\E(X_1)| \leq x\}})$ be the truncated variance of $X_1$.
One easily verifies that $n\sigma^2(a_n) = a_n^2(1+o(1))=n\hat\ell(a_n)(1+o(1))$.

With these specified we recall the precise statement for large deviations
for i.i.d.\ sequences.

\begin{thm}\label{thm-LD}
Assume~\eqref{eq-rvbt} with $\alpha\in (0,2]$. 
Let $N:\N\to\R_{+}$ be so that $N(n)/a_n\to\infty$, as $n\to\infty$.
 Then the following hold as $n\to\infty$:
 
 If $\alpha\in (0,2)$ then
  \begin{equation}
\label{eq:LDthmiid}
 \cP(S_n-b_n>N)=n \cP(X_1>N)(1+o(1)),\quad \cP(S_n-b_n < -N)=n \cP(X_1 <  -N)(1+o(1)).
\end{equation}

If $\alpha=2$ then
\begin{equation}
\label{eq:LDthmiidalpha2}
 \cP(S_n-b_n > N)=
\left( \bar\Phi\left(\frac{N}{a_n}\right)+n\cP(X_1 > N) \right) (1+o(1))
\end{equation}
and a similar statement holds for $\cP(S_n-b_n\le -N)$.
\end{thm}

Theorem~\ref{thm-LD} for $\alpha\in (0,1)\cup (1,2)$ has been known since the works
of~\cite{heyde, Nag1, Nag2}. The more difficult (general) case $\alpha=1$ has only very recently been dealt with in~\cite[Theorem 2.1]{Berger19}.
The statement for $\alpha = 2$ is contained in~\cite{Roz90}: see ~\cite[Theorem 2.1]{BBY},  contained within~\cite[Theorem 6 ]{Roz90}. We remark that~\cite[Condition 2.4 of Theorem 2.1]{BBY}
in Theorem~\ref{thm-LD} holds as soon as $\cP(X_1 > x) \sim c \cP(X_1 < -x)$ for some $c \geq 0$ and $a_n^2 = n\sigma^2(a_n)$: see ~\cite[Examples 2.2 and A.10]{BBY}.

\begin{rmk}\label{rmk:alpha2}
 To illustrate the role of $\bar\Phi$, we can take for example $\ell(x) \equiv 1$, so $\hat \ell(x) = \log x$, and
 $N \sim a_n \log N$. Then the term with $\bar\Phi(N/a_n)$ dominates the other term $n \cP(X_1 > N)$ by approximately a factor $\log N$.
\end{rmk}

We are mainly interested in obtaining a version
of Theorem~\ref{thm-LD} for a  class of deterministic dynamical systems. We postpone the case $\alpha=1$ to later work.

A version of Theorem~\ref{thm-LD} for a class of dependent random variables
satisfying certain clustering and limit conditions has been obtained in~\cite{mik}.
To our knowledge, this is the only previous work in  which a version `close' to the general form of Theorem~\ref{thm-LD}
for dependent random variables has been considered; by `close' we mean that~\eqref{eq:LDthmiid}
is shown to hold in a certain range
of $N/a_n\to\infty$ as opposed to the whole region of $n$ and $N$ so that $N/a_n\to\infty$.
It is not clear to us how to verify the limit conditions in~\cite[Theorem 3.1]{mik}
for the class of dynamical systems considered here and in this work we use a different method of proof.

The somewhat related problem of stable local large deviations (LLD)\footnote{A stable LLD result is concerned 
with bounds on $\cP(S_n-b_n\in (N-h, N+h))$, for $h>0$, 
as $N/a_n\to\infty$.}
for deterministic dynamical systems (including Gibbs-Markov maps) has been treated in~\cite{MTejp}. An optimal LLD has been very recently obtained in~\cite{MPT} for (much more challenging) dynamical systems (with very heavy dependencies) of physical interest known as  Lorentz gases.
We note that stable LLD does not imply stable LD.
Even for i.i.d.\ random variables, in the generality of~\eqref{eq-rvbt},
stable LLD does not imply\footnote{As shown in~\cite{CaravennaDoney, Berger19}, under~\eqref{eq-rvbt} with $\alpha\in (0,2)$ (see also~\cite{MTejp} for a different proof and for the case $\alpha=2$), for any $h>0$, there exists $C>0$ so that $\cP(S_n-b_n\in (x-h, x+h))\le C\frac{n}{a_n}\frac{\ell(|x|)}{1+|x|^\alpha}$, 
for all $n\ge 1$ and for all $x\in\R$.} Theorem~\ref{thm-LD}; though under 
ideal smoothness of the tail probabilities~\eqref{eq-rvbt}, this is in fact possible (see~\cite[Theorem 2.4]{Berger19}). For deterministic dynamical systems, the transition from the stable LLD obtained in~\cite{MTejp, MPT} to
a version of Theorem~\ref{thm-LD} seems impossible.
As we shall explain in the sequel, obtaining a version 
of Theorem~\ref{thm-LD} for dynamical systems is in many respects much more delicate than stable LLD. In this paper we restrict ourselves to  the class of dynamical systems treated in~\cite{MTejp}.

Throughout the rest of the paper, we let $N:\N\to\R_{+}$ be so that 
$N(n) \to \infty$ as $n \to \infty$ and $N(n)/a_n\to\infty$, as $n\to\infty$.

We state the main result and refer to Theorem~\ref{thm-LDGM} below for a more general version.
\begin{thm}\label{thm_informal}
 Let $(T,\Omega,\mu)$ be a probability-preserving mixing Gibbs-Markov map.
Let $v:\Omega\to\R$ be a measurable observable with 
$\int_\Omega v^2\,d\mu=\infty$ and suppose that $v$ satisfies~\eqref{eq-rvbt} (with $\mu$ instead of $\cP$).
Assume further that $v$ is locally Lipschitz and satisfies a `nice'
technical condition (more precisely,~\eqref{eq: extrass} stated in
Subsection~\ref{sec-GM}).

Write $v_n=\sum_{j=0}^{n-1} v\circ T^j$ for the ergodic sum of $v$.  Assume that 
$v_n$ satisfies~\eqref{eq:stlaw} (with $v_n$ instead of $S_n$) with $a_n$ and $b_n$ as in~\eqref{eq:hatell} and~\eqref{eq:bn}.

Let $\delta>0$ be arbitrarily small and set
\[
 D(x)=\begin{cases}
       (\log x)\,\ell(x)x^{-\alpha}, & \text{ if }\alpha\in (0,1),\\
       \,x^{-(\alpha-\delta)}, & \text{ if }\alpha\in (1,2].
      \end{cases}\]
      
 Take any $N = N(n)$ such that $N/a_n \to \infty$.
 Then, for $\alpha\in(0,1)\cup(1,2)$  the following holds as $n\to\infty$:
\[
\left|\mu(v_n>N)- n \mu(v>N)\right|=O
\left(D(N)\right)+o\left(n \ell(N) N^{-\alpha}\right)
\]
and a similar statement holds for $\mu(v_n < -N)$.

Also, for $\alpha=2$, as $n\to\infty$,
\[
\left|\mu(v_n>N)- n \mu(v>N)\right|=O(D(N))+ O\left(n \log N\, N^{-2}\right)+O\left(n \hat\ell(N) N^{-2}\right)\]
and a similar statement holds for $\mu(v_n\le -N)$.
~\end{thm}

The corollary below gives the range of $N=N(n)$ for which
$\mu(v_n>N)\sim n \mu(v>N)$ holds when $\alpha\in (1,2)$.
However, in the case $\alpha=2$, notice that if $\ell\equiv 1$
then the asymptotic equality fails (see Remark~\ref{rmk:alpha2}).
For $\alpha\in (1,2)$, the probability that $v_n>N$ is equivalent to the probability 
that there is exactly one jump larger than $N$, that is that there exists exactly one $j=0,\ldots, n-1$ so that $v\circ T^j>N$.
\begin{cor} Assume the setup of Theorem~\ref{thm-LDGM}.
Then $\mu(v_n>N)\sim n \mu(v>N)$ holds in the following cases.
\begin{itemize}
 \item For $\alpha\in (0,1)$ and $N\in (a_n, e^{an})$, with $a$ small and independent of $n$.
 \item For $\alpha\in (1,2)$ and $N\in (a_n, n^{1/\delta})$
 for any $\delta>0$. 
\end{itemize}

\end{cor}

While in this work we focus on observables taking values in $\R$, we believe that at the expense of cumbersome notation throughout, Theorem~\ref{thm_informal}
can be extended to $\R^d$-valued observables.

Although the proof of Theorem~\ref{thm_informal} below is quite technical, the main steps can be summarised as follows:
\begin{enumerate}
 \item Rephrase Theorem~\ref{thm-LD} for $\alpha\in (0,1)\cup (1,2)$
 (and~\eqref{eq:LDthmiidalpha2} for $\alpha=2$) for i.i.d.\ sequences in terms of characteristic functions: see Section~\ref{sec:strat}.
 \item Decompose the Fourier transform (analogue of characteristic function)
 for the ergodic sum $v_n=\sum_{j=0}^{n-1} v\circ T^j$ into the characteristic function of an i.i.d.\ sequence
 and `some good' quantities: see Subsections~\ref{subsec:strategy}
 and~\ref{subsec:decomp}.
 \item Estimate the various integrals appearing (from the `good' quantities) in the analytic expression of $\mu(v_n-b_n>x)$ via `modulus of continuity' type arguments (in the sense of~\cite[Chapter 1]{Katzn}): see Subsection~\ref{subsec:main abstr}.
 \item Put item 1. and item 3. together to conclude.
 
\end{enumerate}

\vspace{-2ex}
\paragraph{Notation}
We use ``big O'' and $\ll$ notation interchangeably, writing $b_n=O(c_n)$ or $b_n\ll c_n$
if there are constants $C>0$, $n_0\ge1$ such that
$b_n\le Cc_n$ for all $n\ge n_0$.
As usual, $b_n=o(c_n)$ means that $\lim_{n\to\infty}b_n/c_n=0$
and $b_n\sim c_n$ means that $\lim_{n\to\infty}b_n/c_n=1$.

\section{Rephrasing Theorem~\ref{thm-LD} in terms of characteristic functions}
\label{sec:strat}

We start with a general lemma, valid whenever $N(n) \to \infty$ as $n \to \infty$.

\begin{lemma}\label{lemY}
Let $(Z_n)_{n \in \N}$ be a sequence of random variables and $Y$ another random variable, defined on the same probability space but 
independent of all the $Z_n$.
Assume that there exist two sequences $N=N(n)$ and
$h_N=h_{N(n)}$ such that $\cP(|Y| > h_N) = o(\cP(Z_n > N))$
and $\cP(Z_n > N \pm h_N) = \cP(Z_n > N)(1+o(1))$
as $n \to \infty$.

Then for $\tilde Z_n = Z_n+Y$ we have
$$
\cP(\tilde Z_n > N) = \cP(Z_n > N)(1+o(1)) \quad \text{ as } \ n \to \infty.
$$
\end{lemma}

\begin{proof}
Since $Z_n$ and $Y$ are independent, we have (for $h = h_{N(n)}$)
\begin{eqnarray*}
 \cP(Z_n > N+h) &=&  \cP(Z_n > N+h,\, |Y| > h) +  \cP(Z_n > N+h ,\, |Y| \leq h) \\
 &\leq& \cP(|Y| > h) +  \cP(\tilde Z_n > N).
\end{eqnarray*}
Therefore
$$
\cP(\tilde Z_n > N) \geq \cP(Z_n > N+h) - \cP(|Y| > h)) = \cP(Z_n > N)(1+o(1)).
$$
For the other inequality, we have
\begin{eqnarray*}
 \cP(Z_n \leq N-h) &=&  \cP(Z_n \leq N-h ,\, |Y| > h) +  \cP(Z_n \leq N-h,\, |Y| \leq h) \\
 &\leq& \cP(|Y| > h) +  \cP(\tilde Z_n \leq N).
\end{eqnarray*}
Therefore,
\begin{eqnarray*}
\cP(\tilde Z_n > N) &=& 1-\cP(\tilde Z_n \leq N) \leq 1-\cP(Z_n \leq N-h) + \cP(|Y| > h) \\
&=& \cP(Z_n > N-h) + \cP(|Y| > h) = \cP(Z_n > N) (1+o(1)).
\end{eqnarray*}
This finishes the proof.
\end{proof}

For  $t\in\R$, let $\Psi(t)=\E(e^{itX_1})$ be the characteristic function of $X_1$.
In this section we translate Theorem~\ref{thm-LD} for  $\alpha\in(0,1)\cup(1,2]$ into a statement on the characteristic function $\Psi(t)$.
Such a translation is captured in equation~\eqref{eq:right} below, which is the starting point for the dynamic setting.
Since $\alpha\ne 1$, with no loss of generality we can assume that $b_n=0$. For $\alpha\in (0,1)$ this is given, while
for $\alpha\in (1,2]$, we simply replace $X_1$ by $X_1-\E(X_1)$.

\begin{prop}\label{prop:bas}
Assume the set-up of Theorem~\ref{thm-LD} with $\alpha\ne 1$.
Let $Y$ be an $L^2$ random variable, independent of $S_n$ (for each $n$), with
real-valued, even and $C^2$ characteristic function $\psi_Y$, supported in $[-\eps,\eps]$ for some 
\footnote{The existence of such a random variable $Y$ on the same probability space and with characteristic function $\psi_Y$ follows
from, for
instance,~\cite[Proposition 3.8]{Gou-asip},~\cite[Proof of Theorem 3]{KeTe21} and~\cite[Footnote 1]{MTejp}.} small $\eps>0$.

Take any $N = N(n)$ such that $a_n = o(N(n))$
and let $g(n)$ be so that $n N=o(g(n))$.
Then

\begin{itemize}
\item[(i)] If $\alpha\in (0,1)\cup (1,2)$, then the first asymptotic equivalence in~\eqref{eq:LDthmiid} holds if and only if
\begin{equation}
\label{eq:right}
 \left|\int_{-\eps}^\eps (e^{-itN}-e^{-it(N+g(n))})
 \psi_Y(t)\,\frac{\Psi(t)^n-n\Psi(t)}{it}\, dt\right|=o\left(n \cP(X_1>N)\right)
\end{equation}
as $n\to\infty$.

\item[(ii)] If  $\alpha=2$ then~\eqref{eq:LDthmiidalpha2} holds if and only if
\begin{equation}
\label{eq:right2222}
 \left|\int_{-\eps}^\eps (e^{-itN}-e^{-it(N+g(n))})
 \psi_Y(t)\,\frac{\Psi(t)^n-n\Psi(t)}{it}\, dt
 - \bar\Phi\left(\frac{N}{a_n}\right) \right|=
 o\left(n\cP(X_1 > N)+\bar\Phi\left(\frac{N}{a_n}\right) \right)
\end{equation}
as $n \to \infty$.
\end{itemize}
\end{prop}

\begin{rmk}
(a)  The second asymptotic equivalence  in~\eqref{eq:LDthmiid}  can be dealt with in the same way
using the fact that $\cP(S_n < -N)=\cP(-S_n > N)$.

(b) If the random variables $X_i$ are $\Z$-valued (so, $S_n$ is $\Z$-valued), then
one could exploit a simpler formula, bypassing the presence of $\psi_Y$ in the statement of Proposition~\ref {prop:bas}.
Indeed,
\[
 \cP(S_n>N)=\sum_{j>N}\cP(S_n=j)=\sum_{j>N} \frac{1}{2\pi} \int_{-\pi}^\pi e^{-itj} \Psi(t)^n\, dt.
\]
\end{rmk}

\begin{proof}
To start with, recall the inversion formulas
$$
 \cP(X_1 \in (N, N+g(n)])=\lim_{T\to\infty}\frac{1}{2\pi }\int_{-T}^T\frac{e^{-itN}-e^{-it(N+g(n))}}{it}\Psi(t)\, dt
$$
and 
$$
 \cP(S_n\in (N, N+g(n)])=\lim_{T\to\infty}\frac{1}{2\pi }\int_{-T}^T\frac{e^{-itN}-e^{-it(N+g(n))}}{it}\Psi(t)^n\, dt,
$$
for any choice of $g(n)$.
We want to allow that $N(n), g(n)\to\infty$ and obtain the desired formula for $\cP(S_n > N)$.
Recall that $n N= o(g(n))$ and compute that
\[\cP(X_1 > N+g(n))=p \ell(N+g(n)) \frac{1}{g(n)^{\alpha}} \frac{1+o(1)}{\left(1+\frac{N}{g(n)}\right)^{\alpha}} =o\left(\frac{n\ell(N+g(n))}{g(n)^{\alpha}}\right)= o(\cP(X_1 > N)).\]
Since $N=o\left(\frac{N+g(n)}{n}\right)$,
\[\cP(S_n > N+g(n)) \le \sum_{j=0}^{n-1} \cP\left(|X_j|>\frac{N+g(n)}{n}\right) = o(n\cP(X_1 > N)).\]
Thus,  given the choice of $g(n)$, we have $\cP(X_1 \in (N, N+g(n)]) = \cP(X_1 > N) (1-o(1))$
and 
\begin{eqnarray*}
\cP(S_n \in (N, N+g(n)]) &=& \cP(S_n > N) - \cP(S_n > N+g(n))\\
&=& \cP(S_n > N)  - o(n\cP(X_1 > N)).
\end{eqnarray*}

Therefore, as $n\to\infty$, 
\begin{equation}
\label{eq:mform}
 \cP(S_n>N)=\lim_{T\to\infty}\frac{1}{2\pi }\int_{-T}^T\frac{e^{-itN}-e^{-it(N+g(n))}}{it}\Psi(t)^n\, dt+ o(n\cP(X_1 > N))
\end{equation}
and
\begin{equation}
\label{eq:mformx1}
 \cP(X_1 >N)=\lim_{T\to\infty}\frac{1}{2\pi }\int_{-T}^T\frac{e^{-itN}-e^{-it(N+g(n))}}{it}\Psi(t)\, dt + o(\cP(X_1 > N)).
\end{equation}

Next, we argue that one can adjust the domain of integration in~\eqref{eq:mform}.
 Let $\tilde S_n=S_n+Y$,
where $Y$ is as in the statement, that is,  an $L^2$ random variable, independent of $S_n$, with
real-valued, even and $C^2$ Fourier transform $\psi_Y$, supported in $[-\eps,\eps]$ for some small $\eps>0$.

The analogues of~\eqref{eq:mformx1} and~\eqref{eq:mform} are:
\begin{equation}
\label{eq:mformx1Y}
 \cP(X_1+Y>N)=\frac{1+o(1)}{2\pi }\int_{-\eps}^\eps\frac{e^{-itN}-e^{-it(N+g(n))}}{it}\psi_Y(t)\Psi(t)\, dt
\end{equation}
and
\begin{equation}\label{eq:mformY}
 \cP(S_n+Y>N)=\frac{1}{2\pi }\int_{-\eps}^\eps\frac{e^{-itN}-e^{-it(N+g(n))}}{it}\psi_Y(t) \Psi(t)^n\, dt+o\left(n\cP\left(X_1+\frac{Y}{n} > N\right)\right).
\end{equation}

To compare the tails of $X_1$ and $X_1+Y$ , and of $S_n$
and $\tilde S_n = S_n+Y$, we will use Lemma~\ref{lemY}
several times, making different choices for $Z_n$
appearing in that lemma.

\textbf{The case $\alpha \in (0,1) \cup (1,2)$}. 
Take $\delta \in (0, 2-\alpha)$.
Choose $h_N = N^{\frac{\alpha+\delta}{2}} = o(N)$, so
$\cP(X_1 > N \pm h_N) = \cP(X_1 > N)(1+o(1))$.
Because $Y$ is $L^2$, the tail
$\cP(|Y| > h_N) = o(h_N^{-2}) =  o(N^{-\alpha-\delta}) =
o(\ell(N) N^{-\alpha}) = o(\cP(X_1 > N))$. 
Therefore Lemma~\ref{lemY} for $Z_n \equiv X_1$ gives
\begin{equation}\label{eqPPY}
\cP(X_1>N) = \cP(X_1+Y > N)(1+o(1))=\frac{1+o(1)}{2\pi }\int_{-\eps}^\eps\frac{e^{-itN}-e^{-it(N+g(n))}}{it}\psi_Y(t)\Psi(t)\, dt.
\end{equation}
Next we want to obtain a version with $S_n$ in~\eqref{eq:mformY}.

\textbf{From~\eqref{eq:LDthmiid} to~\eqref{eq:right}}.
If $S_n$ satisfies~\eqref{eq:LDthmiid},
i.e., $\cP(S_n>N) = n\cP(X_1 > N)(1+o(1))$, then
the previous computation and choice of $h_N$ give
also $\cP(S_n > N+h_N) = \cP(S_n > N)(1+o(1))$
and
$\cP(|Y| > h_N) = o(\cP(X_1 > N)) = o(\cP(S_n> N))$.
Therefore Lemma~\ref{lemY} for $Z_n = S_n$
and $\tilde S_n = S_n + Y$ give
$$
\cP(S_n>N) = \cP(\tilde S_n>N)(1+o(1)).
$$

\textbf{From~\eqref{eq:right} to~\eqref{eq:LDthmiid}}.
Note that
$\cP\left(\frac{|Y|}{n} > h_N\right)<\cP(|Y| > h_N)=o(\cP(X_1 > N))$.

If~\eqref{eq:right} holds, then using~\eqref{eq:mformY}, and also \eqref{eqPPY} together with
\eqref{eq:mformx1Y}, we obtain that
$\cP(\tilde S_n > N) = n \cP(X_1 > N)(1+o(1))$.
Again using Lemma~\ref{lemY} with $Z_n = \tilde S_n$, $S_n = \tilde S_n-Y$ and $h_N = N^{\frac{\alpha+\delta}{2}} = o(N)$, we can conclude that
\begin{equation}\label{eqSSP}
\cP(\tilde S_n > N) = \cP(S_n > N)(1+o(1)).
\end{equation}
Hence, after subtracting \eqref{eq:mformx1Y}
from~\eqref{eq:mformY}  and apply~\eqref{eqPPY} and~\eqref{eqSSP},
we obtain that
\begin{align*}
\left|\cP(S_n>N)-n \cP(X_1>N)\right|
 &= \frac{1+o(1)}{2\pi} \left|\int_{-\eps}^\eps (e^{-itN}-e^{-it(N+ g(n))})
 \psi_Y(t)\,\frac{\Psi(t)^n-n\Psi(t)}{it} \, dt\right|\\
 &+ o(n\cP(X_1 > N))= o\left(n \cP(X_1>N)\right).
\end{align*}
as  $n, N \to \infty$ so that $a_n= o(N)$.

Next we made the adjustments for \textbf{the case $\alpha = 2$}.
Let $m(N)$ be so that $\cP(|Y| > N)=m(N) N^{-2}$.
So, $m(N) = o(\ell(N))$.
The choice of $m(N)$ ensures that there is a function $\tilde m(N) \to \infty$ such that also
$\tilde m(N)^2 m(N/\tilde m(N)) = o(\ell(N))$.
Now take $h_N = N/\tilde m(N) = o(N)$, so $\cP(X_1 > N+h_N) = \cP(X_1 > N) (1+o(1))$.
Moreover
$$
\cP\left(|Y|/n > h_N\right)<\cP(|Y| > h_N) \ll N^{-2} \tilde m(N)^2 m(N / \tilde m(N))
= o(N^{-2} \ell(N)) = o(\cP(X_1 > N)).
$$
By the same argument used in obtaining~\eqref{eqPPY}, $\cP(X_1+Y > N) = \cP(X_1 > N) (1+o(1))$
and
\begin{equation}\label{eqXa2}
 \cP(X_1 > N) = \frac{1+o(1)}{2\pi} \int_{-\eps}^\eps
 \frac{e^{-itN}-e^{-it(N+ g(n))} }{it} \,
 \psi_Y(t)\, \Psi(t) \, dt.
\end{equation}
The same choice of $h_N$ gives
$\cP(S_n > N+h_N) = \cP(S_n > N) (1+o(1))$
and $\cP(|Y| > h_N) = o(\cP(X_1 > N) = o(\cP(S_n > N)$.
Therefore Lemma~\ref{lemY} gives $\cP(\tilde S_n > N) = \cP(S_n>N) (1+o(1))$.
Combining these estimates with~\eqref{eq:mformY} gives
\begin{equation}\label{eqSa2}
 \cP(S_n  > N) = \frac{1+o(1)}{2\pi} \int_{-\eps}^\eps
 \frac{e^{-itN}-e^{-it(N+ g(n))} }{it} \,
 \psi_Y(t)\, \Psi(t)^n \, dt+o(n\cP(X_1 > N)).
\end{equation}
Now we insert the tail~\eqref{eq:LDthmiidalpha2}, subtract $\bar\Phi(N/a_n)$ and~\eqref{eqXa2} $n$ times. This gives
$$
\left|\int_{-\eps}^\eps (e^{-itN}-e^{-it(N+g(n))})
 \psi_Y(t)\,\frac{\Psi(t)^n-n\Psi(t)}{it}\, dt
 - \bar\Phi\left(\frac{N}{a_n}\right) \right|=
 o\left(n\cP(X_1 > N)+\bar\Phi\left(\frac{N}{a_n}\right) \right),
$$
as required.
For the converse, \eqref{eq:LDthmiidalpha2}
follows directly from the above formula together with
\eqref{eqXa2} and~\eqref{eqSa2}.
\end{proof}

For use below in the set-up of dependent random variables (arising in the context of dynamical systems)
we record the following consequence of the proof of Proposition~\ref{prop:bas}.

\begin{prop}\label{prop:dyneq} Let $(v^j)_{j\ge 0}$ be a sequence of random variables, not necessarily independent, but identically distributed, on some probability
space $(\Omega,\mu)$. Suppose that $v^j$ satisfies~\eqref{eq-rvbt} (with $v^j$ instead of $X_j$ and $\mu$ instead of $\cP$) and write $v_n=\sum_{j=0}^{n-1} v^j$. Assume that 
$v_n$ satisfies~\eqref{eq:stlaw} with $v_n$ instead of $S_n$ and $a_n$ and $b_n$ as in~\eqref{eq:hatell} and~\eqref{eq:bn}.

Let $Y$ be an $L^2$ random variable, independent of $v_n$ (for each $n$), with
real-valued, even and $C^2$ Fourier transform $\psi_Y$, supported in $[-\eps,\eps]$ for some 
\footnote{In the dynamical setting $\eps$ will be fixed at Fact~\ref{fact:chfact1}.} small $\eps>0$.

Take any $N = N(n)$ such that $a_n = o(N(n))$
and let $g(n) = N(n)^{1+\eps_0}$ with $\eps_0>0$.

Moreover, suppose that there exists $W(N,n)$ so that $n \cP(X_1>N)=O\left(W(N,n)\right)$ and so that

\begin{itemize}
\item[(i)] If $\alpha\in (0,1)\cup (1,2)$, as $n\to\infty$,
\begin{equation}
\label{eq:right10}
 \left|\int_{-\eps}^\eps (e^{-itN}-e^{-it(N+g(n))})
 \psi_Y(t)\,\frac{\E_\mu(e^{itv_n})-n\E_\mu(e^{itv^0})}{it}\, dt\right|=O\left(W(N,n)\right).
\end{equation}

\item[(ii)] If  $\alpha=2$, as $n \to \infty$,
\begin{align}
\label{eq:right222210}
 \left|\int_{-\eps}^\eps (e^{-itN}-e^{-it(N+g(n))})
 \psi_Y(t)\,\frac{\E_\mu(e^{itv_n})-n\E_\mu(e^{itv^0})}{it}\, dt
 - \bar\Phi\left(\frac{N}{a_n}\right) \right|&=O\left(W(N,n)\right)\\ \nonumber &+o\left(\bar\Phi\left(\frac{N}{a_n}\right) \right).
\end{align}

\end{itemize}

Then
\begin{itemize}
\item[(i)] If $\alpha\in (0,1)\cup (1,2)$, as $n\to\infty$,
\begin{equation}
\label{eq:right1010}
 \left|\mu(v_n>N)- n \mu(v^0>N)\right|=O\left(W(N,n)\right).
\end{equation}

\item[(ii)] If  $\alpha=2$, as $n \to \infty$,
\begin{equation}
\label{eq:right22221010}
\left|\mu(v_n>N)- n \mu(v^0>N)\right|=O\left(W(N,n)\right) +o\left(\bar\Phi\left(\frac{N}{a_n}\right) \right).
\end{equation}

\end{itemize}

\end{prop}

\begin{proof}
First, note that the only place where we required independence inside the proof of Proposition~\ref{prop:bas}
was where we translated the content of Theorem~\ref{thm-LD} with $\alpha\ne 1$ in terms of $ \left|\int_{-\eps}^\eps (e^{-itN}-e^{-it(N+g(n))})
 \psi_Y(t)\,\frac{\Psi(t)^n-n\Psi(t)}{it}\, dt\right|$. As explained below, the same proof (for the converse part) with $\Psi(t)^n$ replaced by $\E_\mu(e^{itv_n})$  yields the conclusion of the current proposition.
 
 First, by the same argument as the one used in obtaining~\eqref{eq:mform}
 and~\eqref{eq:mformx1}, we have 
 \[
 \mu(v_n>N)=\lim_{T\to\infty}\frac{1}{2\pi }\int_{-T}^T\frac{e^{-itN}-e^{-it(N+g(n))}}{it}\E_\mu(e^{itv_n})\, dt+ o\left(n\mu(v^0 >N)\right)
\]
and 
\(
 \mu(v^0 >N)=\lim_{T\to\infty}\frac{1}{2\pi }\int_{-T}^T\frac{e^{-itN}-e^{-it(N+g(n))}}{it}\E_\mu(e^{itv^0})\, dt.
\)
Also, analogously to~\eqref{eq:mformx1Y} and~\eqref{eq:mformY},
\(
 \mu(v^0+Y>N)=\frac{1}{2\pi }\int_{-\eps}^\eps\frac{e^{-itN}-e^{-it(N+g(n))}}{it}\psi_Y(t)\E_\mu(e^{itv^0})\, dt,
\)

\[
 \mu(v_n+Y>N)=\frac{1}{2\pi }\int_{-\eps}^\eps\frac{e^{-itN}-e^{-it(N+g(n))}}{it}\psi_Y(t) E_\mu(e^{itv_n})\, dt+ o\left(n\mu(v^0 >N)\right)
\]
and
$$
\left|\mu(v_n+Y>N)- n \mu(v^0>N)\right|- o\left(n\mu(v^0 >N)\right)=
\begin{cases}
 O\left(W(N,n)\right) & \text{ if }\alpha\in (0,1)\cup (1,2),\\
 O\left(W(N,n)\right) +o\left(\bar\Phi\left(\frac{N}{a_n}\right) \right)
    & \text{ if }\alpha = 2.
\end{cases}
$$

It remains to argue that $ \mu(v_n+Y>N)= \mu(v_n>N)(1+o(1))$.
 
 If  $\alpha\in (0,1)\cup (1,2)$, we proceed as in obtaining~\eqref{eqSSP},
 by applying Lemma~\ref{lemY} with $Z_n = \tilde S_n$, $S_n = \tilde S_n-Y$ and $h_N = N^{\frac{\alpha+\delta}{2}} = o(N)$.
 If $\alpha=2$,  we proceed as in the argument used in obtaining~\eqref{eqSa2} with the same choice of $h_N$ as there (and $S_n$ replaced by $v_n$).
\end{proof}

\section{Abstract set-up for dynamical systems}

Let $T:\Omega\to \Omega$ be a measure-preserving map
on a probability space $(\Omega,\mu)$.
Let $v:\Omega\to\R$ be a measurable observable with $\int_\Omega v^2\,d\mu=\infty$.
We fix $\alpha\in(0,1)\cup(1,2]$ throughout and
assume
\begin{itemize}
\item[(H1)] 
$\mu(v>x)=p\ell(x) x^{-\alpha}(1+o(1)),\quad \mu(v \le -x)=q\ell(x) x^{-\alpha}(1+o(1)$
\end{itemize}
as $x\to\infty$, where
$\ell$ is slowly varying and $p+q=1$.

Recall $\ell$ and $\hat \ell$ from \eqref{eq:hatell}. Throughout, we shall work with
 \begin{equation}\label{eq:ell0}
  \ell_0 = \begin{cases}
            \ell & \text{ if } \alpha \in (0,1)\cup (1,2);\\
            \hat\ell  & \text{ if } \alpha = 2.
           \end{cases}
 \end{equation}

Let $R:L^1\to L^1$ be the transfer operator for $T$ defined via the formula
$\int_{\Omega} R f\,g\,d\mu = \int_{\Omega}f\,g\circ T\,d\mu$.
Given $t\in\R$, define the perturbed operator $R(t):L^1\to L^1$ by $R(t)f=R(e^{it v}f)$. 

We assume that there is a Banach
space $\cB\subset L^\infty$ containing the constant
functions, with norm $\|\cdot\|$ satisfying $|\phi|_\infty\le \|\phi\|$ for
$\phi\in\cB$, such that
\begin{itemize}
\item[(H2)] There exist $\eps>0$, $C>0$ such that 
for all $|t|,|h|\in B_\eps(0)$,
\begin{itemize}
\item[(i)] 
$\|R(t)\|\le C$ for $\alpha\in(0,1)\cup (1, 2]$ and
$\|R'(t)\|\le C$ for $\alpha\in(1,2]$.
\item[(ii)] 
If $\alpha\in(0,1)$, then $\|R(t+h)-R(t)\|\le C |h|^{\alpha}\ell\Big(\frac{1}{|h|}\Big)$.

\item[(iii)] If $\alpha\in(1,2]$, then $\|R(t+h)-R(t)-ih\, R'(0)\|\le C|h|\,
(|t|^{\alpha'-1}+|h|^{\alpha'-1})$ for any $\alpha'\in (1,\alpha)$.

\item[(iv)] If $\alpha\in(1,2]$, then
 $\|R'(t+h)-R'(t)\|\le C|h|^{\alpha-1}\ell_0\left(\frac{1}{|h|}\right)$.
 \end{itemize}
\end{itemize}

Item (iii) is maybe not as natural as an estimate on $\|R(t+h)-R(t)-ih\, R'(t)\|$,
but it is what is required in the proofs (it is crucial for the proofs of items (ii) and (iii) of Proposition~\ref{prop:dec}), and its validity is checked in Section~\ref{sec:appl}.
Since $R(0)=R$ and $\cB$ contains constant functions, $1$ is an eigenvalue of $R(0)$. 
We assume:
\begin{itemize}\item[(H3)]
The eigenvalue $1$ is simple,
and the remainder of the spectrum of $R(0):\cB\to\cB$ is contained in a disk of radius less than $1$.
\end{itemize}

Throughout, write $v_n=\sum_{j=0}^{n-1} v\circ T^j$. It is known, see~\cite{AD01}, that under hypotheses (H1)--(H3), distributional convergence 
\(
 (v_n-b_n)/a_n\to_d Y_\alpha
\)
to an $\alpha$-stable random variable $Y_\alpha$ holds. Here $a_n$ and $b_n$ are 
defined in~\eqref{eq:hatell} and~\eqref{eq:bn}.

With no loss of generality, when $\alpha\in (1,2]$ we 
assume $\int_{\Omega} v\, d\mu=0$. Hence, $b_n=0$ for $\alpha\in (0,1)\cup (1,2]$. With this specified, we state the main result 
in the abstract set-up.

\begin{thm}
~\label{thm-LDGM} Suppose that (H1)--(H3) hold. Let $\delta>0$ be arbitrarily small and set
\[
 D(x)=\begin{cases}
       (\log x)\,\ell(x)x^{-\alpha}, & \text{ if }\alpha\in (0,1),\\
       \,x^{-(\alpha-\delta)}, & \text{ if }\alpha\in (1,2].
      \end{cases}\]      
 Take $N = N(n)$ such that $N/a_n \to \infty$. 
 Then, as $n\to\infty$,

\begin{itemize}\item[(i)]  If $\alpha\in(0,1)\cup(1,2)$, then
\(
\left|\mu(v_n>N)- n \mu(v>N)\right|=O
\left(D(N)\right)+o\left(n \ell(N) N^{-\alpha}\right)
\) 
and a similar statement holds for $\mu(v_n\le -N)$.

\item[(ii)] If $\alpha=2$, then
\(
\left|\mu(v_n>N)- n \mu(v>N)\right|=O(D(N))+ O\left(n \log N\, N^{-2}\right)+O\left(n \hat\ell(N) N^{-2}\right).
\)
\end{itemize}
\end{thm}

\subsection{Strategy of the proof}
\label{subsec:strategy}

In the proof below of Theorem~\ref{thm-LDGM}  we 
shall proceed  by reasoning as  in Section~\ref{sec:strat}, proving the statement on the right tail. 
More precisely, as in Section~\ref{sec:strat},
we shall work with the inversion formula for distribution functions.

As recalled in Subsection~\ref{subsec:decomp} below,
the relevant characteristic function $\E_\mu(e^{itv_n})$
of the partial sum of dependent quantities $\{v\circ T^j\}_{j\ge 0}$ 
decomposes into the characteristic function of the partial sum of i.i.d.\ random variables $\{\hat v_j\}_{j\ge 0}$,
which we refer to as $\Psi(t)^n$ (see Subsection~\ref{subsec:lps}), and `some other' quantities.

As recalled in Subsection~\ref{subsec:lps}, $\E_\mu(e^{itv})=\E_\mu(e^{it\hat v_j})=\Psi(t)$ for all $t$.
This together with a fact recorded inside the proof of Proposition~\ref{prop:bas}, namely~\eqref{eqPPY} (with $X_j=\hat v_j$), ensures that  given  $\eps>0$ small enough,
\begin{align}\label{eq:vuse}
  \mu(v> N)(1+o(1)) &=
 \frac{1}{2\pi} \lim_{n\to\infty} \int_{-\eps}^\eps\frac{e^{-itN}-e^{-it(N+g(n))}}{it}\psi_Y(t)\, \Psi(t)\,dt,
 \end{align}
 where $\Psi_Y$ is as in the statement of Proposition~\ref{prop:bas}.
 
 By Proposition~\ref{prop:dyneq} (with $v^j=v\circ T^j$) and~\eqref{eq:vuse}, the conclusion of Theorem~\ref{thm-LDGM} follows once we
 show that for some $\eps>0$,
\begin{equation}
\label{eq:rightdyn55}
 \left|\int_{-\eps}^\eps (e^{-itN}-e^{-it(N+g(n))})
 \psi_Y(t)\,\frac{\E_\mu(e^{itv_n})-n\Psi(t)}{it}\, dt\right|,
\end{equation}
is $O\left(D(N)\right)+o\left(n \ell_0(N) N^{-\alpha}\right)$ when $\alpha\in(0,1)\cup(1,2)$
and, respectively, 
\begin{equation}
\label{eq:rightdynalpha255}
 \left|\int_{-\eps}^\eps (e^{-itN}-e^{-it(N+g(n))})
 \psi_Y(t)\,\frac{\E_\mu(e^{itv_n})-n\Psi(t)}{it}\, dt\right|,
\end{equation}
is 
$O(D(N))+ O\left(n \log N\, N^{-2}\right)+O\left(n \hat\ell(N) N^{-2}\right)$ if $\alpha=2$. Here we have used that $\bar\Phi\left(N/a_n\right)\ll (N/a_n)^{-2}\ll n\hat\ell(N)N^{-2}$.

Writing $\E_\mu(e^{itv_n})=\Psi(t)^n+ (\E_\mu(e^{itv_n})-\Psi(t)^n)$ and using the `re-write'  recorded in Proposition~\ref{prop:bas}, equations~\eqref{eq:rightdyn55}
and~\eqref{eq:rightdynalpha255} follow as soon as we show that

\begin{equation}
\label{eq:rightdyn}
 \left|\int_{-\eps}^\eps (e^{-itN}-e^{-it(N+g(n))})
 \psi_Y(t)\,\frac{\E_\mu(e^{itv_n})-\Psi(t)^n}{it}\, dt\right|
\end{equation}
is $O\left(D(N)\right)+o\left(n \ell_0(N) N^{-\alpha}\right)$ when $\alpha\in(0,1)\cup(1,2)$
and
\begin{equation}
\label{eq:rightdynalpha2}
 \left|\int_{-\eps}^\eps (e^{-itN}-e^{-it(N+g(n))})
 \psi_Y(t)\,\frac{\E_\mu(e^{itv_n})-\Psi(t)^n}{it}\, dt-\bar\Phi\left(N/a_n\right)\right|
\end{equation}
is 
$O(D(N))+ O\left(n \log N\, N^{-2}\right)+O\left(n \hat\ell(N) N^{-2}\right)$ if $\alpha=2$.

Since $\E_\mu(e^{itv_n})=\int_\Omega R(t)^n \, 1\, d\mu$, we take the next subsection 
 to decompose $\int_\Omega R(t)^n \, 1\, d\mu$ into $\Psi(t)^n$
 and `good' quantities. The idea is to isolate $\Psi(t)^n$ in the expression of 
$\int_\Omega R(t)^n \, 1\, d\mu$ so as to be able to use the equivalent of~\eqref{eq:right}.

The meaning of `good' quantities will become clear in Subsection~\ref{subsec:main abstr}
below, where we complete the proof of Theorem~\ref{thm-LDGM}.
At this stage, we can point out that the proof uses 
a 'modulus of continuity' argument (and also a derivative in the case $\alpha\in (1,2]$). Although, this type of argument has some similarity
with the ones used in~\cite{MTejp}, the current arguments are much more delicate.
To carry out a `modulus of continuity' type argument, it is crucial that we obtain enough decay in $t$ in the expression of $\int_\Omega R(t)^n \, 1\, d\mu-\Psi(t)^n$.
This is necessary to counteract the effect of the division by $t$
in the integrand in~\eqref{eq:rightdyn}.
The details are postponed to Subsection~\ref{subsec:main abstr}.

\subsection{Decomposing $\int_\Omega R(t)^n \, 1\, d\mu$ into $\Psi(t)^n$
and `good' quantities.}
\label{subsec:decomp}

The main result of this section is Proposition~\ref{prop:dec}~stated at the end of the present subsection.

We first recall some general facts under (H1)--(H3), some of them recently
clarified in~[Subsection 3.2]\cite{MTejp}.

By (H2) and (H3), there exists $\eps>0$ and a continuous family $\lambda(t)$ of simple
eigenvalues of $R(t)$ for $|t|\le 3\eps$ with
$\lambda(0)=1$.  
The associated spectral projections $P(t)$, $|t|\le3\eps$, form a continuous family of bounded linear operators on $\cB$.
Moreover, there is a continuous family of linear operators $Q(t)$ on $\cB$
and constants $C>0$, $\delta_0\in(0,1)$ such that for $|t|\le3\eps$
\begin{align}
\label{eq-sp}
R(t)=\lambda(t) P(t)+Q(t) \qquad \text{ and } \qquad  \|Q(t)^n\|\le C\delta_0^n.
\end{align}
Let $\zeta(t)=\frac{P(t) 1}{\int P(t) 1\, d\mu}$ be the normalised eigenvector corresponding to $\lambda(t)$.

\begin{lemma} \label{lemma-PQ} Assume (H2).
There exists $\eps>0$ such that
the properties of $R(t)$ listed in (H2) are inherited by
$P(t)$, $Q(t)$, $\lambda(t)$ and $\zeta(t)$
for all $|t|,\,|t+h|\le3\eps$.
\end{lemma}
Lemma~\ref{lemma-PQ} holds
with a simplified version of (H2)(iii) obtained by moving $ih R'(0)$ to the other side of $\ll$-sign); namely, for $\alpha \in (1,2]$,
\begin{equation}\label{eq:H2iiisimplified}
\|R(t+h)-R(t)\| \ll |h|. 
\end{equation}
The full strength of (H2)(iii) is exploited in Lemma~\ref{lemma:ht} below.

\subsubsection{A consequence of (H2)(iii) for the operators $P$ and $Q$}
\begin{lemma}
 \label{lemma:ht} Let $\alpha\in (1,2]$. 
 There exists $C>0$ so that the following hold for all $|t|,|t+h|\in B_{3\eps}(0)$ and for any $\alpha'\in (1,\alpha)$.
 \begin{itemize}
  \item[(i)] $\|P(t+h)-P(t)-i\,h\, P'(0)\|\le C|h|\, (|t|^{\alpha'-1}+|h|^{\alpha'-1})$. A similar statement holds for $\zeta(t)=\frac{P(t) 1}{\int P(t) 1\, d\mu}$.
  
  \item[(ii)] $\|Q(t+h)-Q(t)-i\, h\, Q'(0)\|\le C|h|\, (|t|^{\alpha'-1}+|h|^{\alpha'-1})$.
 \end{itemize}
\end{lemma}

\begin{proof}
\textbf{(i)} 
 Recall that $\lambda(t)$ is an isolated eigenvalue in the spectrum of $R(t)$ for every
$|t|\le 3\eps$. Hence, for any $\delta\in (0,3\eps)$,
$P(t)=\frac{1}{2\pi\, i}\int_{|\xi-1|=\delta}(\xi-R(t))^{-1}\, d\xi$ and thus,
$P'(t)=\frac{1}{2\pi\, i}\int_{|\xi-1|=\delta}(\xi-R(t))^{-1}R'(t)(\xi-R(t))^{-1}\, d\xi$.
With these specified we see that
\begin{align*}
 2\pi\, i\, (P(t+h)-P(t))&=
 \int_{|\xi-1|=\delta}(\xi-R(t+h))^{-1}(R(t+h)-R(t))(\xi-R(t))^{-1}\, d\xi\\
 &=
 \int_{|\xi-1|=\delta}(\xi-R(t+h))^{-1}(R(t+h)-R(t)-ih\, R'(0))\,(\xi-R(t))^{-1}\, d\xi\\
 &\quad +ih\,\int_{|\xi-1|=\delta}(\xi-R(t+h))^{-1}R'(0)(\xi-R(t))^{-1}\,d\xi
 =I(t, h)+J(t,h).
\end{align*}

By (H2)(iii), $\|I(t, h)\|\ll |h|\, (|t|^{\alpha'-1}+|h|^{\alpha'-1})$, for any $\alpha'\in (1,\alpha)$. Also,
\begin{align*}
 J(t,h)&=ih\,\int_{|\xi-1|=\delta}(\xi-R(0))^{-1}R'(0) (\xi-R(t))^{-1}\, d\xi\\
 &+ih\,\int_{|\xi-1|=\delta}\left((\xi-R(t+h))^{-1}-(\xi-R(0))^{-1}\right)R'(0)(\xi-R(t))^{-1} d\xi\\
 &=ih\,\int_{|\xi-1|=\delta}(\xi-R(0))^{-1}R'(0) (\xi-R(0))^{-1}\, d\xi\\
 &+ih\,\int_{|\xi-1|=\delta}\left((\xi-R(t+h))^{-1}-(\xi-R(0))^{-1}\right)R'(0)(\xi-R(t))^{-1}\, d\xi\\
 &+ih\,\int_{|\xi-1|=\delta}(\xi-R(0))^{-1}R'(0)\left((\xi-R(t))^{-1}-(\xi-R(0))^{-1}\right) d\xi\\
 &=i\,h\,(2\pi\, i) P'(0)+ih(J_1(t,h)+J_2(t,h)).
\end{align*}
By \eqref{eq:H2iiisimplified}, i.e., the simplified version of (H2)(iii), 
we have $\|R(t+h)-R(0)\| \leq  \|R(t) - R(0)\|  + \| R(t+h) - R(t)\|\ll |t|+|h|$ and thus, $\|J_1(t,h)\|,\,\|J_2(t,h)\| \ll |h|\, (|t|+|h|)$.

Altogether, $P(t+h)-P(t)-i\,h\, P'(0)=(2\pi\, i)^{-1}I(t, h)+(2\pi\, i)^{-1}\,h\, (J_1(t,h)+J_2(t,h))$ and each term satisfies the claimed estimate, concluding the proof of the statement on $P(t)$.

The statement on $\zeta(t)=\frac{P(t) 1}{\int P(t) 1\, d\mu}$  follows directly from the definition and the statement on $P(t)$.

\textbf{(ii)} We shall use that $Q(t)=R(t)(I-P(t))$ (here $I$ is the identity operator). So, $Q'(t)=(R(t)(I-P(t)))'=R'(t)(I-P(t))-R(t)P'(t)$  and thus,
\begin{equation}
\label{eq:qpr0}
 Q'(0)=R'(0)-\,R(0) P'(0)-\,R'(0)P(0).
\end{equation}

With this specified, we compute that 
\begin{align*}
 Q(t+h)-Q(t)&=R(t+h)-R(t)+R(t)P(t)-R(t+h)P(t+h)\\
 &=R(t+h)-R(t)
 +R(t)(P(t)-P(t+h))- (R(t+h)-R(t))P(t+h)\\
 &=\left(R(t+h)-R(t)-ih\,R'(0)\right)-R(t)\left(P(t+h)-P(t)-ih\, P'(0)\right)\\
 &\quad -\left(R(t+h)-R(t)-ih\,R'(0)\right)P(t+h)\\
 &\quad +ih\,R'(0)-ih\,R(t) P'(0)-ih\,R'(0)P(t+h).
\end{align*}
Using (H2)(iii) and item \textbf{(i)} above,
\begin{align*}
 \left\| Q(t+h)-Q(t)-ih\left(R'(0)-\,R(t) P'(0)+\,R'(0)P(t+h)\right)\right\|\ll|h|\, 
 (|t^{\alpha'-1}+|h|^{\alpha'-1}),
\end{align*}
for any $\alpha'\in (1,\alpha)$.
But, by (H2)(iii) and~\eqref{eq:qpr0},
\begin{align*}
 ih\left(R'(0)-\,R(t) P'(0)-\,R'(0)P(t+h)\right)&=
 ih\left(R'(0)-\,R(0) P'(0)-\,R'(0)P(0)\right)+D(t,h)\\
 &=ih\,Q'(0)+D(t,h),
\end{align*}
where $\|D(t,h)\|\ll|h|\, (|t|+|h|)$.~\end{proof}

\subsubsection{Some properties of $\lambda$ and details on $\Psi$}
\label{subsec:lps}

Let $\Psi(t)$ be the characteristic function of
the i.i.d.\ random variables $\hat v_j:\Omega\to\R$  so that $\hat v_j$  and $v\circ T^j$ have the same distribution.

Similarly to $v$,  we take $\int_{\Omega}\hat v_j\, d\mu=0$ (when $\alpha>1$).
Since $\hat v_j$ satisfies (H1), 
 \begin{equation}
 \label{eq:psiid}
 1-\Psi(t)=c_\alpha\ell_0\left(\frac{1}{|t|}\right)|t|^ \alpha(1 +o(1))\text{ as }t\to 0.
 \end{equation}
Equation~\eqref{eq:psiid} has been established in~\cite{GnedenkoKolmogorov, IL} (see also~\cite[Theorem 5.1]{AD01}). Also, as established in these works
 (in particular,~\cite{AD01})  and clarified in~\cite[Lemma 2.1]{MTejp},
 
 \begin{fact}\label{fact:chfact1}
There exist constants $\eps,\,c>0$, such that
\[
|\Psi(t) | \leq \exp\left\{- c |t|^{\alpha} \ell_0\left(\frac{1}{|t|}\right)\right\}
\quad\text{for all $t\in B_{3\eps}(0)$.}
\]
\end{fact}
 Throughout the rest of the paper we fix $\eps$
 so that both Fact~\ref{fact:chfact1} and equation~\eqref{eq-sp} hold.
 
 It is known that there exist $C>0$ so that for all $|t|,|t+h|\in B_{3\eps}(0)$,
 \begin{equation}
 \label{eq:psicont}
 |\Psi(t+h)-\Psi(t)|\le \begin{cases}
    C|h|^\alpha\ell\Big(\frac{1}{|h|}\Big), &  \alpha\in(0,1),\\
    C|h|, & \alpha\in(1,2].                         
                        \end{cases}
 \end{equation}
Furthermore, if $\alpha\in (1,2]$, then $\Psi$ is differentiable and writing
$\Psi'$ for its derivative,
\begin{equation}
 \label{eq:psider}
 |\Psi'(t)|\le C\ell_0\left(\frac{1}{|t|}\right)|t|^ {\alpha-1},
 \end{equation}
 for some $C>0$. For a precise reference for the validity of~\eqref{eq:psicont}
 and~\eqref{eq:psider},
 see, for instance,~\cite[Lemma 2.2]{MTejp}.
 
 \begin{lemma} \label{lemma-lambda}
Let
\(
 V(t)=\lambda(t)-\Psi(t).
\)
There exist $C, C'$ so that the following hold for all $|t|, |t+h|\in B_{3\eps}(0)$:

 \begin{itemize}
  \item[(i)] If $\alpha\in (0,1)$, then $|V(t)|\le C|t|^{2\alpha}\ell\Big(\frac{1}{|t|}\Big)^2$
    and 
  $|V(t+h)-V(t)|\le C'|h|^\alpha\ell\Big(\frac{1}{|h|}\Big)|t|^\alpha\ell\Big(\frac{1}{|t|}\Big)$.
  
  If $\alpha\in (1, 2]$, then $|V(t)|\le C|t|^{2}$ and
  $|V(t+h)-V(t)|\le C'|h|\,(|t|+|h|)$.
  \item[(ii)] 
  If $\alpha\in (1, 2]$, then $V$ is differentiable and writing  $V'$ for its derivative, $|V'(t)|\le C|t|$ and 
  \(
  |V'(t+h)-V'(t)|\le C|h|^{\alpha'-1}\ell_0\left(\frac{1}{|h|}\right)(|t|+|h|) +C'|h|,
  \) for any $\alpha'\in (1,\alpha)$.
 \end{itemize}
\end{lemma}
 
 Before proceeding to the proof we note that
 Lemma~\ref{lemma-lambda} together with~\eqref{eq:psiid} and~\eqref{eq:psider}
 implies that
 \begin{equation}
 \label{eq:lam}
 1-\lambda(t)=c_\alpha\ell_0\left(\frac{1}{|t|}\right)|t|^ \alpha(1 +o(1))\text{ and for }
 \alpha\in(1,2],\,
 |\lambda'(t)|\le C\ell_0\left(\frac{1}{|t|}\right)|t|^ {\alpha-1},
 \end{equation}
 for some $C>0$.

\begin{pfof}{Lemma~\ref{lemma-lambda}}
Write
\begin{align*} 
\lambda(t) & =\int R(t)\zeta(t)\,d\mu=
\int_{\Omega} R(t)1\,d\mu+\int(R(t)-R(0))(\zeta(t)-\zeta(0))\,d\mu 
\\ & =
\int_{\Omega} e^{it v}\,  d\mu+\int_{\Omega} (R(t)-R(0))(\zeta(t)-\zeta(0))\, d\mu=\Psi(t)+V(t).
\end{align*}
This gives the (well known) decomposition of $\lambda$ into $\Psi$ and $V$ in the statement of the lemma.
We still need to clarify the properties of $V$.

Since $\cB\subset L^\infty$, item~(i) for the case $\alpha\in(0,1)$ follows immediately from (H2)(ii) and Lemma~\ref{lemma-PQ}.
The second part of item~(i) for the case $\alpha\in(1,2]$ follows similarly
using that by \eqref{eq:H2iiisimplified}, $\|R(t+h)-R(0)\|\ll |t|+|h|$.

For item~(ii), we first note that
\[
V'(t)=\int_{\Omega} (R(t)-R(0)) \zeta'(t)\, d\mu+\int_{\Omega} R'(t)\,(\zeta(t)-\zeta(0))\, d\mu.
\]
Lemma~\ref{lemma-PQ} and the fact together with $\|R(t)-R(0)\|\ll |t|$
imply that $|V'(t)|\ll |t|$.
 Next, compute that 
\begin{align*}
  V'(t+h)-V'(t)&=\int_{\Omega} (R(t+h)-R(t)) \zeta'(t+h)\, d\mu
  +\int_{\Omega} (R(t)-R(0)) (\zeta'(t+h)-\zeta'(t))\, d\mu\\
  &+\int_{\Omega} (R'(t+h)-R'(t))\,(\zeta(t)-\zeta(0))\, d\mu
 +\int_{\Omega} R'(t+h)\,(\zeta(t+h)-\zeta(t))\, d\mu\\
 &=I_1(t,h)+I_2(t,h)+I_3(t,h)+I_4(t,h).
  \end{align*}
 Recall $\|R'(t)\|, |\zeta'(t)|\ll 1$ and $\|R(t+h)-R(0)\|\ll |t|+|h|$. 
 These together with (H2)(iv) and Lemma~\ref{lemma-PQ} give that
 $|I_2(t,h)|, |I_4(t,h)|\ll |h|^{\alpha-1}\ell_0\left(\frac{1}{|h|}\right)\,(|t|+|h|)$.
 Regarding $I_1$, compute that
\begin{align*}
  I_1(t,h) &=ih\int_{\Omega} R'(0) \zeta'(t+h)\, d\mu
 +\int_{\Omega} (R(t+h)-R(t)-ih\,R'(0)) \zeta'(t+h)\, d\mu\\
 &=ih\int_{\Omega} R'(0) \zeta'(0)\, d\mu+ih\int_{\Omega} R'(0) (\zeta'(t+h)-\zeta'(0))\, d\mu\\
  &+\int_{\Omega} (R(t+h)-R(t)-ih\,R'(0)) \zeta'(t+h)\, d\mu.
  \end{align*}
  
  Similarly,
  \begin{align*}
  I_4(t,h) &=ih\int_{\Omega} R'(0) \zeta'(0)\, d\mu+ih\int_{\Omega} (R'(t+h)-R'(0))\zeta'(0)\, d\mu\\
  &+\int_{\Omega} R'(t+h)\,(\zeta(t+h)-\zeta(t)-ih\,\zeta'(0))\, d\mu.
  \end{align*}
  By (H2)(iii), (H2)(iv) and Lemma~\ref{lemma-PQ},
  \[
   |I_1(t,h)-ih\int_{\Omega} R'(0) \zeta'(0)\, d\mu|\ll |h|\,(|t|^{\alpha'-1}+|h|^{\alpha'-1}),
  \]
   for any $\alpha'\in (1,\alpha)$.
 By (H2)(iv) (together with Lemma~\ref{lemma-PQ}) and Lemma~\ref{lemma:ht} (i),
 \[
  |I_4(t,h)-ih\int_{\Omega} R'(0) \zeta'(0)\, d\mu|\ll |h|(|t|^{\alpha'-1}+|h|^{\alpha'-1}),
  \]
  for any $\alpha'\in (1,\alpha)$.
Hence,
\[
 |I_1(t,h)+I_4(t,h)|\ll |h|+|h|(|t|^{\alpha'-1}+|h|^{\alpha'-1}),
\]
for any $\alpha'\in (1,\alpha)$, ending the proof. ~\end{pfof}

For use in the proofs to follow we recall one more property of $\lambda$,
when $\alpha\in(1,2]$; this is a refined version of the continuity properties
recorded in Lemma~\ref{lemma-PQ}.~(This type of estimate appears inside~\cite[Proof of Theorem 3.2]{MTejp}.)
 
 \begin{lemma}\label{cor-lambda2}
If $\alpha\in(1,2]$, then there exists $C>0$ so that for all $t,t+h\in B_{3\eps}(0)$ with $|h|\le \frac{|t|}{2}$,
$$
|\lambda(t+h)-\lambda(t)|\le C|h|\,(|t|^{\alpha-1}+|h|^{\alpha-1})\ell_0\left(\frac{1}{|t|+|h|}\right).
$$

A similar statement holds for $\Psi$.
\end{lemma}
\begin{proof}
 By the mean value theorem, $\lambda(t+h)-\lambda(t)= h\lambda'(t^*)$,
 for some $t^*\in [t,t+h]$ and the conclusion follows. 
\end{proof}

\subsubsection{A first decomposition of $\int_{\Omega} R(t)^n\, d\mu$
  into $\Psi(t)^n$ and `good' quantities}

\begin{lemma}
\label{lemma:dec} Let 
\[
\hat V(t, n)=\int_{\Omega} R(t)^n 1\, d\mu-\Psi(t)^n -\int_{\Omega} Q(t)^n 1\, d\mu.
\]
The following hold for all $n\ge 1$ and 
all $|t|,|h|\in B_\eps(0)$ with $|h|\le |t|/2$.

\begin{itemize}
 \item[(i)] If $\alpha\in (0,1)$, then
   \begin{eqnarray*}
   |\hat V(t, n)|\!\!\!\! &\ll& \!\!\!\! \left(|t|^{\alpha}\ell\Big(\frac{1}{|t|}\Big)
   +\,n |t|^{2\alpha}\ell\Big(\frac{1}{|t|}\Big)^2\right)\,|\Psi(t)^n|
   \\
 |\hat V(t+h,n)-\hat V(t, n)|\!\!\!\! &\ll&\!\!\!\! \left(|h|^{\alpha}\ell\Big(\frac{1}{|h|}\Big)
 +\, n|h|^{\alpha}\ell\Big(\frac{1}{|h|}\Big)\,|t|^{\alpha}\,\ell\left(\frac{1}{|t|}\right)\right)\,|\Psi(t)^n|.
    \end{eqnarray*}  
    
  \item[(ii)]  If $\alpha\in (1,2]$, then for any $\alpha'\in(1,\alpha)$,
  \[
   \left|\hat V(t, n)-it\,\lambda(t)^n\int_{\Omega} P'(0)\, 1\, d\mu\right|\ll  \left(|t|^{\alpha'}+\,n |t|^2\right)\,|\Psi(t)^n|.
    \]
    
\item[(iii)] If $\alpha\in (1,2]$, then for any $\alpha'\in(1,\alpha)$,
  \begin{align*}
   &\left|\hat V(t+h,n)-\hat V(t, n)-ih\,\lambda(t)^n\int_{\Omega} P'(0)\, 1\, d\mu\right|\\
   &\ll \left(|h|\, |t|^{\alpha'-1}\,
   +  n\,|h|\, |t|^{\alpha-1}\,\ell_0\left(\frac{1}{|t|}\right)\right)\,|\Psi(t)^n|.
    \end{align*}
\item[(iv)] If $\alpha\in (1,2]$, then $\hat V$ is differentiable in $t$ and writing 
$\hat V'(t,n)$ for its derivative,
  \begin{align*}
     \left|\hat V'(t, n)-\lambda(t)^n\int_{\Omega} P'(0)\, 1\, d\mu\right|&\le \left( |t|^{\alpha-1}\ell_0\left(\frac{1}{|t|}\right)
     + n\,|t|
     +n^2\,|t|^{\alpha+1}\ell_0\left(\frac{1}{|t|}\right)\right)\,|\Psi(t)^n|
     \end{align*}
   and
 \begin{align*}
 \Big|\hat V'(t+h,n)- &\hat V'(t, n)-
 ih\,n\,\lambda(t)^{n-1}\lambda'(t) \int_{\Omega}P'(0))1 \, d\mu\Big|\ll
 |h|^{\alpha-1}\ell_0\left(\frac{1}{|h|}\right)\,|\Psi(t)^n|\\
 &+n\left(|t|\,|h|^{\alpha'-1}\ell_0\left(\frac{1}{|h|}\right)
 +|h|
 \right)|\Psi(t)^n|\\
 &+n^2 \left(|t|^{2\alpha-1}|h|\ell_0\left(\frac{1}{|t|}\right)^2+|t|^2|h|^{\alpha-1}\ell_0\left(\frac{1}{|h|}\right)+ |h|\, |t|^\alpha\ell_0\left(\frac{1}{|t|}\right)\right)\,|\Psi(t)^n|\\
 &+n^3 |t|^{2\alpha}|h|\ell_0\left(\frac{1}{|t|}\right)^2\,|\Psi(t)^n|.
\end{align*}
\end{itemize}
\end{lemma}

\begin{proof}
By~\eqref{eq-sp},
\begin{align*}
 \int_{\Omega} R(t)^n 1\, d\mu &-\int_{\Omega} Q(t)^n 1\, d\mu =\lambda(t)^n\int_{\Omega} P(t) 1\, d\mu\\
 &=\lambda(t)^n\int_{\Omega} P(0) 1\, d\mu
 +\lambda(t)^n\int_{\Omega} (P(t)-P(0)) 1\, d\mu\\
 &=\Psi(t)^n +(\lambda(t)^n-\Psi(t)^n)
 +\lambda(t)^n\int_{\Omega} (P(t)-P(0)) 1\, d\mu.
\end{align*}
Hence,
\begin{align}\label{eq;hatv}
 \nonumber \hat V(t, n)&=(\lambda(t)^n-\Psi(t)^n)
 +\lambda(t)^n\int_{\Omega} (P(t)-P(0))1\, d\mu\\
 &= A(t, n)+B(t, n).
\end{align}
Throughout the rest of the proof we will use that by Lemma~\ref{lemma-lambda},
$|\lambda(t)^n|\ll |\Psi(t)^n|$.

\textbf{Item (i).} Note that by Lemma~\ref{lemma-PQ}
and (H2)(ii), $|B(t, n)|\ll |t|^{\alpha}\ell\Big(\frac{1}{|t|}\Big)|\Psi(t)^n|$.
Next, using Lemma~\ref{lemma-PQ} and  (H2)(ii),
\[
 |B(t+h,n)-B(t, n)|\le |h|^{\alpha}\ell\Big(\frac{1}{|h|}\Big)|\Psi(t)^n|
 +|\lambda(t+h)^n-\lambda(t)^n||t|^{\alpha}\ell\Big(\frac{1}{|t|}\Big).
\]
Note that $\lambda(t+h)^n-\lambda(t)^n=(\lambda(t+h)-\lambda(t))\sum_{j=0}^{n-1}\lambda(t+h)^j\lambda(t)^{n-j}$.  This together with Lemma~\ref{lemma-PQ}
implies that
\(
 |\lambda(t+h)^n-\lambda(t)^n|\ll n |h|^{\alpha}\ell\Big(\frac{1}{|h|}\Big)|\Psi(t)^n|.
\)
Hence, 
\[
 |B(t+h,n)-B(t, n)|\le |h|^{\alpha}\ell\Big(\frac{1}{|h|}\Big)(n |t|^{\alpha}\ell\Big(\frac{1}{|t|}\Big) +1)\,|\Psi(t)^n|.
\]
It remains to deal with $A(t, n)$ in~\eqref{eq;hatv}.
Given $V(t) = \lambda(t) - \Psi(t)$ as in Lemma~\ref{lemma-lambda}, $A(t, n)=V(t)\sum_{j=0}^{n-1}\lambda(t)^j\Psi(t)^{n-j-1}$. This together with the properties of $V$ in
Lemma~\ref{lemma-lambda} (i) implies that
\(
 |A(t, n)|\ll n |t|^{2\alpha}\ell\Big(\frac{1}{|t|}\Big)^2|\Psi(t)^n|.
\)
To deal with the continuity properties of $A(t,n)$, we first compute that
\begin{align}\label{eq:lpsth}
 &(\lambda(t+h)^{n}-\Psi(t+h)^{n})-(\lambda(t)^{n}-\Psi(t)^{n})\\
 \nonumber &=V(t+h)\sum_{j=0}^{n-1}\lambda(t+h)^{j}\Psi(t+h)^{n-j-1}
 -V(t)\sum_{j=0}^{n-1}\lambda(t)^{j}\Psi(t)^{n-j-1}\\
 \nonumber &=(V(t+h)-V(t))\sum_{j=0}^{n-1}\lambda(t+h)^{j}\Psi(t+h)^{n-j-1}\\
 \nonumber &+V(t)\sum_{j=0}^{n-1}\left(\lambda(t+h)^{j}-\lambda(t)^{j}\right)\Psi(t+h)^{n-j-1}
 +V(t)\sum_{j=0}^{n-1}\lambda(t)^{j}\left(\Psi(t+h)^{n-j-1}-\Psi(t)^{n-j-1}\right)
 \end{align}
 Using~\eqref{eq:lpsth}, we compute that
\begin{align*}|A(t+h ,n)-A(t, n)|&\ll I_1+I_2+I_3,
\end{align*}
where
\begin{align*}
 I_1&=|V(t+h)-V(t)|\sum_{j=0}^{n-1}|\lambda(t+h)^j|\, |\Psi(t+h)^{n-j-1}|
 \ll n|h|^{\alpha}\ell\Big(\frac{1}{|h|}\Big)\,|t|^{\alpha}\,\ell\Big(\frac{1}{|t|}\Big)|\Psi(t+h)^n|,
\end{align*}
\[
 I_2=|V(t)|\sum_{j=0}^{n-1}|\lambda(t+h)^j-\lambda(t)^j|\,|\Psi(t+h)^{n-j-1}|\ll 
 n|h|^{\alpha}\ell\Big(\frac{1}{|h|}\Big)\,|t|^{2\alpha}\,\ell\Big(\frac{1}{|t|}\Big)^2|\Psi(t+h)^n|,
\]
where we have used again Lemma~\ref{lemma-PQ} and Lemma~\ref{lemma-lambda},
and
\[
 I_3=|V(t)|\sum_{j=0}^{n-1}|\lambda(t)^j|\,|\Psi(t+h)^{n-j-1}-\Psi(t)^{n-j-1}|
 \ll n|h|^{\alpha}\ell\Big(\frac{1}{|h|}\Big)\,|t|^{2\alpha}\,\ell\Big(\frac{1}{|t|}\Big)^2|\Psi(t)^n|,
\]
where we have used~\eqref{eq:psicont} as well.

\textbf{Item (ii).}
We continue from~\eqref{eq;hatv}. First,
$|A(t,n)|\ll n|V(t)||\Psi(t)^n|$ and using
Lemma~\ref{lemma-lambda} (i) (the statement on $V$ for the case $\alpha\in (1,2]$),
$|A(t,n)|\ll n|t|^2|\Psi(t)^n|$.

Next, we rewrite $B(t,n)$ in a convenient way:
\begin{align*}
\label{eq:b1}
 \nonumber B(t,n)=it\lambda(t)^n\int_{\Omega} P'(0)1\, d\mu
 +\lambda(t)^n\int_{\Omega} (P(t)-P(0)-itP'(0))1\, d\mu.
 \end{align*}
 By Lemma~\ref{lemma:ht}(i), $\left|\lambda(t)^n\int_{\Omega} (P(t)-P(0)-itP'(0))1\, d\mu\right|
 \ll |t|^{\alpha'}|\Psi(t)^n|$ for any $\alpha'\in (1,\alpha)$, and the conclusion follows.

\textbf{Item (iii).}
We start from~\eqref{eq;hatv}. First,
\begin{align*}
|A(t+h,n)-A(t,n)|&\ll|\lambda(t+h)^n-\lambda(t)^n|
+|\Psi(t+h)^n-\Psi(t)^n|\\
&\ll n |\lambda(t+h)-\lambda(t)|\,|\lambda(t)^n|
+n |\Psi(t+h)-\Psi(t)|\,|\Psi(t)^n|\\
&\ll n\,|h|\,|t|^{\alpha-1}\, \ell_0\left(\frac{1}{|t|}\right)\,|\Psi(t)^n|,
\end{align*}
where we have used Lemma~\ref{cor-lambda2}.

It remains to deal with $B$ defined in~\eqref{eq;hatv}.
\begin{align*}
 B(t+h, n)- B(t, n)&=\lambda(t)^n\int_{\Omega} (P(t+h)-P(t))1\, d\mu
 +(\lambda(t+h)^n-\lambda(t)^n)\int_{\Omega} (P(t+h)-P(0))1\, d\mu\\
 &=J_1(t, h, n)+J_2(t, h, n).
\end{align*}
First, by Lemma~\ref{lemma-PQ} and
Lemma~\ref{cor-lambda2},
\[
 |J_2(t, h, n)|\ll n|\lambda(t+h)-\lambda(t)|\,|\lambda(t)^n|\,|t+h|\ll n\,
 |h|\,|t|^{\alpha}\,  \ell_0\left(\frac{1}{|t|}\right)|\Psi(t)^n|.
\]
 Now 
\begin{align*}
J_1(t, h, n)=ih\,\lambda(t)^n\int_{\Omega} P'(0)\,1\, d\mu + 
\lambda(t)^n\int_{\Omega} (P(t+h)-P(t)-ihP'(0))1\, d\mu.         
\end{align*}
Using Lemma~\ref{lemma:ht}(i), we obtain that for any $\alpha'\in (1,\alpha)$,
\begin{align*}
\left|J_1(t, h, n)-ih\, \lambda(t)^n\int_{\Omega} P'(0)\,1\, d\mu\right|\ll 
|h|\, (|t|^{\alpha'-1} + |h|^{\alpha'-1})\,|\Psi(t)^n|,
\end{align*}
ending the proof of (iii).

\textbf{Item (iv).} Differentiating in~\eqref{eq;hatv},
\begin{align}\label{eq;hatv222}
 \nonumber \hat V'(t, n)&=n(\lambda(t)^{n-1}\lambda'(t)-\Psi(t)^{n-1}\Psi'(t))
 +n\lambda(t)^{n-1}\lambda'(t)\int_{\Omega} (P(t)-P(0))1\, d\mu\\
 &\quad +\lambda(t)^n\int_{\Omega} P'(t) 1\, d\mu
=A_0(t, n)+A_1(t, n)+A_2(t, n).
\end{align}
Recalling from Lemma~\ref{lemma-lambda} that $V(t) = \lambda(t) - \Psi(t)$
we compute that
\begin{align}\label{eq:a0}
 \nonumber A_0(t, n)&=n\lambda'(t)(\lambda(t)^{n-1}-\Psi(t)^{n-1})
 +n\Psi(t)^{n-1}(\lambda'(t)-\Psi'(t))\\
 &=n\lambda'(t)(\lambda(t)^{n-1}-\Psi(t)^{n-1})+n\Psi(t)^{n-1}V'(t).
\end{align}
It follows from equation~\eqref{eq:lam} and Lemma~\ref{lemma-lambda}(i) and (ii) that
\begin{align*}
 |A_0(t,n)&\ll n\,|t|^{\alpha-1}\ell_0\left(\frac{1}{|t|}\right)\,n\, |V(t)|\,|\Psi(t)^n|
 +n\, |V'(t)|\,|\Psi(t)^n|\\
 &\ll n^2\,|t|^{\alpha+1}\ell_0\left(\frac{1}{|t|}\right)\,|\Psi(t)^n|
 +n |t|\,|\Psi(t)^n|.
\end{align*}

Regarding $A_1$, using equation~\eqref{eq:lam} (statement on the derivative) and
Lemma~\ref{lemma-PQ}, we obtain $|A_1(t, n)|\ll n|t|^{\alpha}\ell_0\left(\frac{1}{|t|}\right)\,|\Psi(t)^n|$.
For $A_2$, write
\begin{align}\label{eq:a2122}
  \nonumber A_2(t, n)&=\lambda(t)^n\int_{\Omega} P'(0) 1\, d\mu+\lambda(t)^n\int_{\Omega} (P'(t)-P'(0)) 1\, d\mu\\
  &=\lambda(t)^n\int_{\Omega} P'(0) 1\, d\mu+A_2^1(t, n).
\end{align}
By
Lemma~\ref{lemma-PQ} and (H2)(iv),
$|A_2^1(t, n)|\ll |t|^{\alpha-1}\ell_0\left(\frac{1}{|t|}\right)\,|\Psi(t)^n|$.
Thus,
\[
 |A_2(t, n)-\lambda(t)^n\int_{\Omega} P'(0) 1\, d\mu|\ll |t|^{\alpha-1}\ell_0\left(\frac{1}{|t|}\right)\,|\Psi(t)^n|.
\]
Putting together the estimates for $A_0, A_1, A_2$,
\begin{align*}
 \left|\hat V'(t, n)-\lambda(t)^n\int_{\Omega} P'(0) 1\, d\mu\right|&
 \ll \left(n^2\,|t|^{\alpha+1}\ell_0\left(\frac{1}{|t|}\right)+n |t|\right)\,|\Psi(t)^n|
 +n\,|t|^{\alpha}\ell_0\left(\frac{1}{|t|}\right) \,|\Psi(t)^n|\\
 &+|t|^{\alpha-1}\ell_0\left(\frac{1}{|t|}\right)\,|\Psi(t)^n|\\
 &\ll \left(|t|^{\alpha-1}\ell_0\left(\frac{1}{|t|}\right)+n |t|
 +n^2\,|t|^{\alpha+1}\ell_0\left(\frac{1}{|t|}\right)\right)\,|\Psi(t)^n|,
\end{align*}
as claimed.

We continue with the continuity properties of $\hat V'$ studying each term in~\eqref{eq;hatv}.

\textbf{Term $A_0$ in~\eqref{eq;hatv222}.}

Using the expression~\eqref{eq:a0}, we compute that

\begin{align*}
 |A_0(t+h,n)-A_0(t, n)|&\ll n^2|\lambda'(t+h)-\lambda'(t)|\,|\Psi(t)^{n-1}|\,|V(t)|\\
 & + n|\lambda'(t)|\Big|(\lambda(t+h)^{n-1}-\Psi(t+h)^{n-1})-
  (\lambda(t)^{n-1}-\Psi(t)^{n-1})\Big|\\
  &+n\,|\Psi(t+h)^{n-1}-\Psi(t)^{n-1}|\,|V'(t)|+n\,|\Psi(t)^{n-1}|\,|V'(t+h)-V'(t)|\\
 &=I_1(t, h, n)+I_2(t, h, n)+I_3( t, h, n)+I_4(t, h, n).
\end{align*}

By Lemma~\ref{lemma-lambda} (ii),
$I_4( t, h, n)\ll n\,|h|^{\alpha'-1}\ell_0\left(\frac{1}{|h|}\right)(|t|+|h|)\, |\Psi(t)^{n-1}|+n\, |h|\,|\Psi(t)^{n-1}|$.

By Lemma~\ref{cor-lambda2} (statement on $\Psi$), Lemma~\ref{lemma-lambda} (ii) and~\eqref{eq:psider},
$I_3(t, h, n)\ll 
n^2\,|h|\,|t|\,\ell_0\left(\frac{1}{|t|+|h|}\right)(|t|^{\alpha-1}+|h|^{\alpha-1})\,|\Psi(t)^n|
\ll n^2\,|h|\,|t|\,\ell_0\left(\frac{1}{|t|+|h|}\right)\,|t|^{\alpha-1}\,|\Psi(t)^n|$.

We continue with  $I_2$. First, recalling equation~\eqref{eq:psider}
and using Lemma~\ref{lemma-lambda} and Lemma~\ref{cor-lambda2},

 \begin{align}\label{eq;estnasty}
  \nonumber \Big|&(\lambda(t+h)^{n}-\Psi(t+h)^{n})-
  (\lambda(t)^{n}-\Psi(t)^{n})\Big|\\
  \nonumber &\ll n|V(t+h)-V(t)|\,|\Psi(t)^{n-1}|
  + n^2|V(t)|\,|\left|\lambda(t+h)-\lambda(t)\right|\,|\Psi(t)^{n-1}|\\
 \nonumber & \quad + n^2|V(t)|\,|\left|\Psi(t+h)-\Psi(t)\right|\,|\Psi(t)^{n-1}|\\
  &\ll n|h|(|t|+|h|)\,|\Psi(t)^{n-1}|+n^2|h|\,|t|^{2}\, (|t|^{\alpha-1}+|h|^{\alpha-1}) \, \ell_0\left(\frac{1}{|t|+|h|}\right)\,|\Psi(t)^{n-1}|.
 \end{align}
By equations~\eqref{eq;estnasty} and~\eqref{eq:lam},
\begin{align*}
 I_2(t, h, n)&\ll n^2|h|\,|t|^{\alpha} \, \ell_0\left(\frac{1}{|t|+|h|}\right)\,|\Psi(t)^{n-1}|
 +n^3|h|\,\,|t|^{2\alpha}\,\ell_0\left(\frac{1}{|t|+|h|}\right)^2\,|\Psi(t)^{n-1}|.
\end{align*}
Finally, by Lemma~\ref{lemma-PQ} and (H2)(iv), $|\lambda'(t+h)-\lambda'(t)|\ll h^{\alpha-1}\ell_0(1/|h|)$.
This together with Lemma~\ref{lemma-lambda} (i) gives that
$I_1(t, h, n)\ll n^2\,|h|^{\alpha-1}|t|^2\ell_0\left(\frac{1}{|t|+|h|}\right)\,|\Psi(t)^n|$.
Putting together the estimates for $I_4, I_3, I_2, I_1$, we obtain
\begin{align}
 \label{eq:a0cont}
 \nonumber |A_0(t+h,n)&-A_0(t, n)|\ll \left(n\,|h|^{\alpha-1}\ell_0\left(\frac{1}{|h|}\right)(|t|+|h|)+n\, |h|\right)\,|\Psi(t)^{n-1}|\\
 \nonumber &+\left(n^2 |h|^{\alpha-1}\ell_0\left(\frac{1}{|h|}\right)
 \,|t|^2\
 +\ n^2\, |h|\,|t|^{\alpha}\,\ell_0\left(\frac{1}{|t|+|h|}\right)\right)\,|\Psi(t)^{n-1}|\\
 &+n^3|h|\,
 |t|^{2\alpha}\, \ell_0\left(\frac{1}{|t|+|h|}\right)^2\,|\Psi(t)^{n-1}|.
 \end{align}

\textbf{Term $A_1$ in~\eqref{eq;hatv222}.}
\begin{align*}
 A_1(t+h,n)&-A_1(t, n)=n\lambda(t)^{n-1}\lambda'(t)\int_{\Omega} (P(t+h)-P(t))1\, d\mu\\
 &+n(\lambda(t+h)^{n-1}-\lambda(t)^{n-1})\lambda'(t)\int_{\Omega} (P(t+h)-P(0))1\, d\mu\\
 &+n\lambda(t+h)^{n-1}(\lambda'(t+h)-\lambda'(t))\int_{\Omega} (P(t+h)-P(0))1\, d\mu\\
 &=J_1(t,h, n)+J_2(t,h, n)+J_3(t,h, n).
\end{align*}
By Lemma~\ref{lemma-PQ} and arguments similar to the ones used in estimating $A_0$ above,
$|J_2(t,h, n)|\ll n^2\,|h|\,|t|^{2\alpha-1}\,\ell_0\left(\frac{1}{|t|+|h|}\right)^2\,|\Psi(t)^n|$
and $|J_3(t,h, n)|\ll n\, |h|^{\alpha-1}\ell_0\left(\frac{1}{|h|}\right)\, |t|\,|\Psi(t)^n|$.

Finally, for any $\alpha'\in (1,\alpha)$,
\begin{align*}
 J_1(t,h, n)&=ih\,n\,\lambda(t)^{n-1}\lambda'(t) \int_{\Omega}P'(0)1\, d\mu+n\lambda(t)^{n-1}\lambda'(t)\int_{\Omega} (P(t+h)-P(t)-ih\, P'(0))1\, d\mu\\
 &=ih\,n\,\lambda(t)^{n-1}\lambda'(t) \int_{\Omega}P'(0)1\, d\mu
 +O\left(n\, |h|\,|t|^{\alpha'+\alpha-2}\,|\Psi(t)^n|\right),
\end{align*}
where in the last line we have used equation~\eqref{eq:lam} and Lemma~\ref{lemma:ht} (i). 

Altogether,
\begin{align}\label{eq:a1contpr}
 \nonumber \Big| & A_1(t+h,n)-A_1(t, n)-ih\,n\,\lambda(t)^{n-1}\lambda'(t) \int_{\Omega}P'(0))1\Big|\\
 &\ll \left(n\, |h|\,|t|^{\alpha'+\alpha-2}+n^2\,|h|\,|t|^{2\alpha-1}\,\ell_0\left(\frac{1}{|t|+|h|}\right)^2
 +n\, |h|^{\alpha-1}\ell_0\left(\frac{1}{|h|}\right)\,|t|\right)|\Psi(t)^n|.
\end{align}

\textbf{Term $A_2$ in~\eqref{eq;hatv222}.}
Here we use the definition of $A_2^1$ in~\eqref{eq:a2122}.
By Lemma~\ref{lemma-PQ} and Lemma~\ref{cor-lambda2},
\begin{align*}
 |A_2^1(t+h,n)-A_2^1(t, n)|&\ll |\lambda(t+h)^n-\lambda(t)^n|\,\|P'(t)\|
+|\lambda(t)^n|\, \|P'(t+h)-P'(t)\|\\
&\ll  n\,|h|\, |t|^{\alpha-1}\ell_0\left(\frac{1}{|t|+|h|}\right) \,|\Psi(t)^n|+|h|^{\alpha-1}\ell_0\left(\frac{1}{|h|}\right)\,|\Psi(t)^n|.
\end{align*}

The second statement of item (iv) follows by putting together the estimate on $A_2^1$ 
with the equations~\eqref{eq:a0cont} and~\eqref{eq:a1contpr}.
More precisely, putting all these together (recorded in the order $A_0, A_1$ and $A_2$)
and recalling~\eqref{eq:a1contpr},
\begin{align*}
 \Big|\hat V'(t+h,n)- &\hat V'(t, n)-
 ih\,n\,\lambda(t)^{n-1}\lambda'(t) \int_{\Omega}P'(0))1 \, d\mu\Big|\ll
 |h|^{\alpha-1}\ell_0\left(\frac{1}{|h|}\right)\,|\Psi(t)^n|\\
 &+n\left(|t|\,|h|^{\alpha'-1}\ell_0\left(\frac{1}{|h|}\right)
 +|h|+
 |t|^{\alpha'+\alpha-2}|h|\ell_0\left(\frac{1}{|t|}\right)\right)|\Psi(t)^n|\\
 &+n^2 \left(|t|^{2\alpha-1}|h|\ell_0\left(\frac{1}{|t|+|h|}\right)^2+|t|^2|h|^{\alpha-1}\ell_0\left(\frac{1}{|h|}\right)+ |h|\, |t|^\alpha\ell_0\left(\frac{1}{|t|+|h|}\right)\right)\,|\Psi(t)^n|\\
 &+n^3 |t|^{2\alpha}|h|\ell_0\left(\frac{1}{|t|+|h|}\right)^2\,|\Psi(t)^n|.
\end{align*}

Using that $|h|\le |t|/2$, we can ignore the term $|t|^{\alpha'+\alpha-2}|h|\ell_0\left(\frac{1}{|t|}\right)$.
More precisely, we see that for any $\eps>0$,
\begin{align*}
 |t|^{\alpha'+\alpha-2}|h|\ell_0\left(\frac{1}{|t|}\right)
 =|t|\, |t|^{\alpha'+\alpha-3}\, |h|^{2-\alpha-\eps}\ell_0\left(\frac{1}{|t|}\right)\, |h|^{\alpha-1+\eps}
 \ll |t|\, |t|^{\alpha'-1-\eps}\ell_0\left(\frac{1}{|t|}\right)|h|^{\alpha-1+\eps}
\end{align*}
Since $\alpha'>1$ and $\eps>0$ is arbitrarily small,  $|t|^{\alpha'-1-\eps}\ell_0\left(\frac{1}{|t|}\right)=o(1)$ as $|t|\to 0$.
Also, $|h|^{\alpha-1+\eps}\ll |h|^{\alpha-1}\ell_0\left(\frac{1}{|h|}\right)$. Thus,
\[
 |t|^{\alpha'+\alpha-2}|h|\ell_0\left(\frac{1}{|t|}\right)
 \ll |t|\,|h|^{\alpha-1}\ell_0\left(\frac{1}{|h|}\right)\ll |t|\,|h|^{\alpha'-1}\ell_0\left(\frac{1}{|h|}\right).
\]
Item \textbf{(iv)} follows.~\end{proof}

\subsubsection{Properties of $Q(t)^n$}

\begin{lemma}
 \label{lemma-Qal} For $n \ge 1$, consider the operator $W(t,n)=Q(t)^n-Q(0)^n$.
 Let $\delta_0$ be as in~\eqref{eq-sp}. Let $|t|,|h|\in B_\eps(0)$.
 The following hold for some $\delta_0 < \delta_1 < \delta_2 < 1$ and for some $C',C''>0$.
 \begin{itemize}
  \item[(i)] If $\alpha \in (0,1)$, then
  \begin{equation*}
   \|W(t, n)\|\le C' \delta_1^n |t|^{\alpha}\ell\Big(\frac{1}{|t|}\Big),\,\,
  \|W(t+h, n)-W(t, n)\|\le C''\delta_1^n\,|h|^{\alpha}\ell\Big(\frac{1}{|h|}\Big).
    \end{equation*}
    
   \item[(ii)]  If $\alpha\in (1,2]$, then for any $\alpha'\in (1,\alpha)$.
   \begin{align*}
  \left\|\left(W(t+h, n)-W(t, n)-ih\,Q(0)^{n-1}\, Q'(0)\right)\,1\right\|\le C'\delta_2^n\,|h|\,
  (|t|^{\alpha'-1}+|h|^{\alpha'-1})+C''\,n |h|(|t|+|h|).
     \end{align*}
    
    \item[(iii)] If $\alpha\in (1,2]$, then $W$ is differentiable and noticing $W'(t,n)=(Q(t)^n)'$,
  \begin{align*}
   \|W'(t+h, n)-W'(t, n)\|=\|(Q(t+h)^n)'-(Q(t)^n)'\|\le C''\delta_1^n\,|h|^{\alpha-1}\ell_0\left(\frac{1}{|h|}\right).
    \end{align*}
\end{itemize}
\end{lemma}

\begin{proof}
 Compute that
 \[
  \|W(t,n)\|=\sum_{j=0}^{n-1}\|Q(t)^j\|\,\|(Q(t)-Q(0))\|\,\|Q(0)^{n-j-1}\|,
 \]
\[
 \|W(t+h, n)-W(t, n)\|\ll\sum_{j=0}^{n-1}\|Q(t+h)^j\|\,\|(Q(t+h)-Q(t))\|\,\|Q(t)^{n-j-1}\|
\]
and
\[
 (Q(t)^n)'=\sum_{j=0}^{n-1}Q(t)^jQ'(t)Q(t)^{n-j-1}.
\]
Using these equations, as well as $\| Q(t+h)-Q(t)\| \ll |t|+|h|$ (which follows from~\eqref{eq:H2iiisimplified} and $\|Q^n(t)\| \leq C\delta_0^n$ from \eqref{eq-sp}, we have
\begin{equation}\label{eq-qq}
 \| Q(t+h) - Q(0) \| \ll \delta_0^n (|t|+|h|).
\end{equation}
 
\textbf{Items (i) and (iii)} are included in~\cite[Proof of Lemma 3.7]{MTejp}; the claimed estimates follow directly from the previous 
three displayed formulas, equation~\eqref{eq-sp}, assumptions (H2)(i) and (H2)(iv), respectively and Lemma~\ref{lemma-PQ}.
 
\textbf{Item (ii)}
Note that $W(t+h, n)-W(t, n)-ih\,(Q(0)^n)'=Q(t+h)^n-Q(t)^n-ih\,(Q(0)^n)'$.
\begin{quote}
Take $M \in \N$ such that $\| Q^M(t) \| \leq C \delta_0^M \leq \frac12$
in \eqref{eq-sp}.
\end{quote}
We will need a two-step induction argument, first for $n \leq M$, and then for $n > M$.
\\[1mm]
{\bf Step 1: $n \leq M$:}
Take the induction hypothesis:
for any $\alpha'\in (1,\alpha)$, there is $C_{ind} > 0$ such that
\begin{align}\label{eq:claaaim}
\nonumber & \left\|Q(t+h)^n-Q(t)^n-ih\,Q(0)^{n-1} Q'(0)- i h
 \Big(Q(0)Q'(0)+Q'(0)Q(0)\Big)\sum_{j=0}^{n-1} Q(0)^j\right\|\\
 &\qquad \qquad \ll  C_{ind}\, n |h| \left( C^n \delta_1^{n-1}\, (|t|^{\alpha'-1}+|h|^{\alpha'-1})  + (|t|+|h|) \right),
\end{align}
for $C$ from \eqref{eq-sp}. We will first prove this statement by induction
but only for $1 \leq n \leq M$.
Afterwards we explain how to adjust the induction for $n > M$
and with a new uniform constant $\hat C_{ind}$ instead of $C_{ind} C^n$.
It is very likely that the second term $n |h|(|t|+|h|)$ can be improved but this estimate suffices for the proof of the main result.


Since  $Q(0) \,1=R(0)(I-P(0))\,1=0$, we have $\sum_{j=1}^{n-1} Q(0)^j 1=0$ and $n\delta_0^n\le \delta_1^n$ for some $\delta_1\in(\delta_0, 1)$,
item \textbf{(ii)} follows.

We show that~\eqref{eq:claaaim} holds using an induction argument.
A straightforward (but lengthy) calculation verifies the statement for $n=2$. To see this,  compute that 
\begin{align*}
 Q(t+h)^2 &-Q(t)^2
 =Q(t+h)(Q(t+h)-Q(t))+(Q(t+h)-Q(t))Q(t)\\
 &=Q(t+h)\left(Q(t+h)-Q(t)-ih\, Q'(0)\right)\\
 &+\left(Q(t+h)-Q(t)-ih\, Q'(0)\right)Q(t)
 +ih\,Q(t+h)Q'(0)+ih\,Q'(0)Q(t)
\end{align*}
So,
\begin{align*}
 (Q(t+h)^2 -Q(t)^2)-ih\,Q(0)Q'(0)&-ih\,Q'(0)Q(0)=Q(t+h)\left(Q(t+h)-Q(t)-ih\,Q'(0)\right)\\
 &+\left(Q(t+h)-Q(t)-ih\, Q'(0)\right)Q(t)\\
 &+ih\,(Q(t+h)-Q(0))Q'(0)+ih\,Q'(0)(Q(t)-Q(0)).
\end{align*}
Using Lemma~\ref{lemma:ht}(ii), (H2)(iii) and Lemma~\ref{lemma-PQ},
\begin{align*}
 \left\|Q(t+h)^2 -Q(t)^2-ih\,Q(0)Q'(0)-ih\,Q'(0)Q(0)\right\|\ll 2\delta_0|h|\, (|t|^{\alpha'-1}+|h|^{\alpha'-1})+|h|\left(|t|+|h|\right),
\end{align*}
for any $\alpha'\in (1,\alpha)$. This verifies the statement for $n=2$. 
To make the assumption~\eqref{eq:claaaim} more clear , we also record $n=3$.
\begin{align*}
 Q(t+h)^{3}-Q(t)^{3}&=Q(t+h)^{2}(Q(t+h)-Q(t))+(Q(t+h)^{2}-Q(t)^{2})Q(t)\\
 &=Q(t+h)^{2}(Q(t+h)-Q(t)-ih Q'(0))\\
 &+(Q(t+h)^{2}-Q(t)^{2}-ih\,Q(0)Q'(0)-ih\,Q'(0)Q(0))Q(t)\\
 &+ih Q(t+h)^{2}Q'(0)+ih \left(Q(0)Q'(0)+Q'(0)Q(0)\right)Q(t)
\end{align*}
So,
 \begin{align*}
 Q(t+h)^{3}-Q(t)^{3}-&ih Q(0)^{2}Q'(0)+ih \left(Q(0)Q'(0)+Q'(0)Q(0)\right)Q(0)\\
 &=Q(t+h)^{2}(Q(t+h)-Q(t)-ih Q'(0))\\
 &+(Q(t+h)^{2}-Q(t)^{2}-ih\,Q(0)Q'(0)-ih\,Q'(0)Q(0))Q(t)\\
 &+ih \left(Q(t+h)^{2}-Q(0)^2\right)Q'(0)\\&+ih \left(Q(0)Q'(0)+Q'(0)Q(0)\right)(Q(t)-Q(0))\\
 &=I_1+I_2+I_3+I_4.
\end{align*}
Using the same same estimates as for $n=2$,
we have $|I_1|\ll \delta_0^2 |h|\, (|t|^{\alpha'-1}+|h|^{\alpha'-1})$.
Using the statement on $n=2$ and the fact that $\|Q(t)\|\ll \delta_0$, $|I_2|\ll 2\delta_0^2\,|h|\, (|t|^{\alpha'-1}+\delta_0\,|h|^{\alpha'-1})+|h|\left(|t|+|h|\right)$,
$|I_3|\ll \delta_0^2|h|(|t|+|h|)$ (using also~\eqref{eq-qq})
and $|I_4|\ll |h|(|t|+|h|)$. Thus,
\begin{align*}
 |I_1+I_2+I_3+I_4|\ll 3\delta_0^2 |h|\, (|t|^{\alpha'-1}+|h|^{\alpha'-1}) +(1+\delta_0)|h|(|t|+|h|).
\end{align*}
Since $1+\delta_0<3$, the statement for $n=3$ is verified.

The general case $n\ge 1$ follows by induction, assuming the statement that~\eqref{eq:claaaim} holds for $n$ and want to prove the hypothesis for $n+1$.
Note that
\begin{align*}
 Q(t+h)^{n+1}&-Q(t)^{n+1}=Q(t+h)^{n}(Q(t+h)-Q(t))+(Q(t+h)^{n}-Q(t)^{n})Q(t)\\
 &=Q(t+h)^{n}\left(Q(t+h)-Q(t)-ih\, Q'(0)\right)\\
 &+\left(Q(t+h)^{n}-Q(t)^{n}-ih\,Q(0)^{n-1} Q'(0)-i h \left(Q(0)Q'(0)+Q'(0)Q(0)\right)\sum_{j=0}^{n-1} Q(0)^j\right)Q(t)\\
 &+ih\,Q(t+h)^nQ'(0)+ih\,Q(0)^{n-1}Q'(0)Q(t)\\
 &+i h \Big(Q(0)Q'(0)+Q'(0)Q(0)\Big)\sum_{j=0}^{n-1} Q(0)^jQ(t).
\end{align*}
Thus,
\begin{align*}
 Q(t+h)^{n+1}&-Q(t)^{n+1}-ihQ(0)^nQ'(0)-i h \Big(Q(0)Q'(0)+Q'(0)Q(0)\Big)\sum_{j=0}^{n-1} Q(0)^j\,Q(0)\\
 &=Q(t+h)^{n}\left(Q(t+h)-Q(t)-ih\, Q'(0)\right)\\
 &+\left(Q(t+h)^{n}-Q(t)^{n}-ih\,Q(0)^{n-1} Q'(0)-i h \Big(Q(0)Q'(0)+Q'(0)Q(0)\Big)\sum_{j=0}^{n-1} Q(0)^j\right)Q(t)\\
 &+ih\left(Q(t+h)^n-Q(0)^n\right)Q'(0)\\
 &+i h \Big(Q(0)Q'(0)+Q'(0)Q(0)\Big)\sum_{j=0}^{n-1} Q(0)^j(Q(t)-Q(0))\\
 &=J_1+J_2+J_3+J_4.
\end{align*}
Now, $|J_1|\ll \delta_0^n |h|\, (|t|^{\alpha'-1}+|h|^{\alpha'-1})$. Using~\eqref{eq:claaaim} and \eqref{eq-sp},
\begin{equation}\label{eq-J2}
|J_2| \leq C C_{ind} n |h| \left(C^n \delta_1^{n-1} \, (|t|^{\alpha'-1}+|h|^{\alpha'-1})+(|t|+|h|)\right).
\end{equation}
Using \eqref{eq-qq}, $|J_3| \ll \delta_0^1|h| \, (|t|+|h|)$.
Since $\|Q(t)-Q(0)\|\ll |t|$ and $\|\sum_{j=0}^{n-1} Q(0)^j\|=O(1)$, we get $|J_4|\ll |h|(|t|+|h|)$. 
The statement on $n+1$ follows by adding these estimates.
\\[1mm]
{\bf Step 2: $n > M$:}
Next, we will do the induction in steps of $M$ iterates.
That is, we consider
$$
\hat R(t) := R^M(t) = \lambda^M(t) P(t) + Q^M(t) =:
\hat \lambda(t) P(t) + \hat Q(t),
$$
and recall that $\|\hat Q\| \leq \frac12$ by the choice of $M$.
The reason why the previous induction is of no use for all $n$ is due to the factor $C^n$, of which we have no control. But using steps of $M$ iterates, $C$
can be replaced by $\frac12$, so if we use the induction hypothesis:
for any $\alpha'\in (1,\alpha)$, there is $\hat C_{ind} > 0$ such that
\begin{align}\label{eq:claaaim2}
\nonumber & \left\|\hat Q(t+h)^n-\hat Q(t)^n-ih\, \hat Q(0)^{n-1} \hat Q'(0)- i h
 \Big(\hat Q(0) \hat Q'(0)+\hat Q'(0) \hat Q(0)\Big)\sum_{j=0}^{n-1} \hat Q(0)^j\right\|\\
 &\qquad \qquad \ll  \hat C_{ind}\, n |h| \left( \delta_2^{n-1}\, (|t|^{\alpha'-1}+|h|^{\alpha'-1})  + (|t|+|h|) \right),
\end{align}
for $\delta_2 := \delta_1^{1/M}$,
the above induction proof works.
Apart from $M$-dependent constants, the estimates of $J_1$, $J_3$ and $J_4$ still hold.
This gives the required result
for all $n = n_0 + jM$ and $j \geq 0$ (where the initial step $j=0$
follows from the first induction for $1 \leq n_0 \leq M$).
\end{proof}

\subsubsection{A final form of the decomposition of $\int_{\Omega} R(t)^n 1\, d\mu$
  into $\Psi(t)^n$ and `good' quantities}

To simplify the statement of Lemmas~\ref{lemma:dec} and \ref{lemma-Qal} further, we record the following facts.
We recall that we assume $\int_{\Omega} v\, d\mu=0$ (when $\alpha>1$).
\begin{lemma}
 \label{lemma:pq0} 
 Let $(Q(0)^n)'=(Q(t)^n)'|_{t=0}$, $n\ge 1$.
 Then
 \begin{equation*}
  \int_{\Omega} P'(0)\, 1\, d\mu=0\text{ and }\int_{\Omega} (Q(0)^n)'\, 1\, d\mu=\int_{\Omega} Q(0)^{n-1}Q'(0)\, 1\, d\mu=0.
 \end{equation*}
\end{lemma}

\begin{proof}
As in the proof of Lemma~\ref{lemma:ht},
\[
 P'(0)\, 1 = \frac{1}{2\pi\, i}\int_{|\xi-1|=\delta}(\xi-R(0))^{-1}R'(0)(\xi-R(0))^{-1}\, 1\, d\xi.
\]
But  $(\xi-R(0))^{-1}\, 1=(\xi-1)^{-1}$
for all $\xi \neq 1$.
Since we also know that $R'(0)\, 1= i R(0) v$, using Fubini's theorem we obtain  that
\begin{align*}
 2\pi\,i\,\int_{\Omega} P'(0)\, 1\, d\mu &=i \int_{\Omega} \int_{|\xi-1|=\delta}(\xi-1)^{-1}(\xi-R(0))^{-1}R(0) v\, d\xi\, d\mu\\
 &
 =i \int_{|\xi-1|=\delta}(\xi-1)^{-1}\sum_{j=0}^\infty \xi^{-j-1}\int_{\Omega} R(0)^j
 R(0) v\, d\mu\, d\xi\\
 &
 = i \int_{|\xi-1|=\delta}(\xi-1)^{-1}\sum_{j=0}^\infty \xi^{-j-1}\int_{\Omega} R(0)^{j+1}
 v\,  d\mu\, d\xi\\
 &=i \int_{|\xi-1|=\delta}(\xi-1)^{-1}\sum_{j=0}^\infty \xi^{-j-1}\int_{\Omega}  v\, d\mu\, d\xi=0,
 \end{align*}
concluding the argument for the first part of the statement.

 Next, recall that $(Q(t)^n)'=\sum_{j=0}^{n-1}Q(t)^jQ'(t)Q(t)^{n-j}$.
 Hence,
 \[
  \int_{\Omega} (Q(0)^n)'\, 1\, d\mu=\int_{\Omega} \sum_{j=0}^{n-1}Q(0)^jQ'(0)Q(0)^{n-j-1}\, 1\, d\mu.
 \]
Since $Q(0)\, 1=0$, $Q(0)^{n-j-1}\, 1=0$, for all $j<n-1$. Thus,
\[
 \int_{\Omega} (Q(0)^n)'\, 1\, d\mu=\int_{\Omega} Q(0)^{n-1}Q'(0)\, 1\, d\mu.
\]
Recall from~\eqref{eq:qpr0} that $Q'(0)\, 1=-R(0)P'(0)\, 1$. As a consequence, 

\begin{align*}
 \int_{\Omega} (Q(0)^n)'\, 1\, d\mu &=-\int_{\Omega} Q(0)^{n-1}R(0)P'(0)\, 1\,d\mu\\
& =-\int_{\Omega} R(0)^{n-1}(I-P(0))R(0)P'(0)\, 1\,d\mu\\
 &=-\int_{\Omega} R(0)^{n}P'(0)\, 1\,d\mu+\int_{\Omega} R(0)^{n-1}P(0)R(0)P'(0)\, 1\,d\mu\\
 &=-\int_{\Omega} P'(0)\, 1\,d\mu+\int_{\Omega}R(0) P'(0)\, 1\,d\mu =0,
\end{align*}
which ends the proof.~\end{proof}

Putting together the last three lemmas we obtain

\begin{prop}
\label{prop:dec} Suppose that (H1)--(H3) hold.
Let 
\(
U(t, n)=\int_{\Omega} R(t)^n 1\, d\mu-
\Psi(t)^n.
\)
Then the following hold for
all $n\ge 1$ and all $|t|,|h|\in B_\eps(0)$ with $|h|\le |t|/2$.

\begin{itemize}
 \item[(i)] For $\alpha\in (0,1)$, 
  \begin{eqnarray*}
   |U(t, n)| \!\!\!\! &\le& \!\!\!\!  \left(C_0|t|^{\alpha}\ell\Big(\frac{1}{|t|}\Big)
   +n |t|^{2\alpha}\ell\Big(\frac{1}{|t|}\Big)^2\right)\,|\Psi(t)^n|,
   \\
 |U(t+h,n)-U(t, n)| \!\!\!\! &\le& \!\!\!\! \left(|h|^{\alpha}\ell\Big(\frac{1}{|h|}\Big)
 + n|h|^{\alpha}\ell\Big(\frac{1}{|h|}\Big)\,|t|^{\alpha}\,\ell\Big(\frac{1}{|t|}\Big)\right)\,|\Psi(t)^n|.
    \end{eqnarray*}
    
    \item[(ii)]
    If $\alpha\in (1,2]$, the following hold for any $\alpha'\in (1,\alpha)$.
  \begin{align*}
   &|U(t, n)|\le \left(|t|^{\alpha'}+n |t|^2\right)\,|\Psi(t)^n|
    \end{align*}
  and
  \begin{align*}
   | &U(t+h,n)-U(t, n)|
 \le |h|\, |t|^{\alpha'-1}\,|\Psi(t)^n| +n\,|h|\, |t|.
  \end{align*}

\item[(iii)] If $\alpha\in (1,2]$, then $U$ is differentiable in $t$ and writing 
$U'(t,n)$ for its derivative,
  \begin{align*}
   U'(t, n)\ll \left( |t|^{\alpha-1}\ell_0\left(\frac{1}{|t|}\right)+  n\,|t|
     +n^2\,|t|^{\alpha+1}\ell_0\left(\frac{1}{|t|}\right)\right)\,|\Psi(t)^n|
   \end{align*}
   and
 \begin{align*}
 \Big|U'(t+h,n)- & U'(t, n)\Big|\ll
 |h|^{\alpha-1}\ell_0\left(\frac{1}{|h|}\right)\,|\Psi(t)^n|\\
 &+n\left(|t|\,|h|^{\alpha'-1}\ell_0\left(\frac{1}{|h|}\right)
 +|h|
 \right)|\Psi(t)^n|\\
 &+n^2 \left(|t|^{2\alpha-1}|h|\ell_0\left(\frac{1}{|t|}\right)^2+|t|^2|h|^{\alpha-1}\ell_0\left(\frac{1}{|h|}\right)+ |h|\, |t|^\alpha\ell_0\left(\frac{1}{|t|}\right)\right)\,|\Psi(t)^n|\\
 &+n^3 |t|^{2\alpha}|h|\ell_0\left(\frac{1}{|t|}\right)^2\,|\Psi(t)^n|.
    \end{align*}
\end{itemize}
\end{prop}

\begin{proof}
In the notation of Lemma~\ref{lemma:dec} and Lemma~\ref{lemma-Qal},
\[
 U(t, n)=\int_{\Omega} R(t)^n 1\, d\mu-
\Psi(t)^n=\hat V(t, n)+\int_{\Omega} W(t, n)\, 1\, d\mu
\]
since $W(t, n)\, 1=Q(t)^{n}\, 1-Q(0)^{n}\, 1=Q(t)^{n}\, 1$
(because $Q(0)^{n}\, 1=0$ for all $n\ge1$).

 \textbf{Item (i)} follows immediately from Lemma~\ref{lemma:dec} (i) and Lemma~\ref{lemma-Qal} (i).
 
 \textbf{Item (ii).}
 By Lemma~\ref{lemma:pq0}, $\int_{\Omega} P'(0)\, 1\, d\mu=\int_{\Omega} (Q(0)^n)'\, 1\, d\mu=\int_{\Omega} Q(0)^{n-1}Q'(0)\, 1\, d\mu=0$.
 
 Using also Lemma~\ref{lemma:dec} (ii),
   \[
   \left|U(t, n)\right|\ll  \left(|t|^{\alpha'}+n |t|^2\right)\,|\Psi(t)^n|,
    \]
  which proves the first statement of item (ii).
  
 We already know that $\int_{\Omega} Q(0)^{n-1}Q'(0)\, 1\, d\mu=0$.
This together with Lemma~\ref{lemma-Qal} (ii) (and recalling $h|\le |t|/2$) implies that
    \[
     \left|\int_{\Omega}\left(W(t+h, n)-W(t, n)\right)\,1\, d\mu\right|\ll \delta_1^n\,|h|\,
  |t|^{\alpha'-1}+n |h|\,|t|.
    \]

Combining this with Lemma~\ref{lemma:dec} (iii),
    \begin{align*}
   \left|U(t+h,n)- U(t, n)\right|&\le \left(|h|\, |t|^{\alpha'-1}
   +  n\,|h|\, |t|^{\alpha-1}\,\ell_0(1/|t|)\right)\,|\Psi(t)^n|
   +n |h|\,|t|\\
   &\ll|h|\, |t|^{\alpha'-1}\,|\Psi(t)^n|+n |h|\,|t|,
\end{align*}
  which proves the second statement of item (ii).

\textbf{Item (iii)} follows immediately from Lemma~\ref{lemma:dec} (iv),  Lemma~\ref{lemma-Qal} (iii) 
and the fact that $\int_{\Omega} P'(0)\, 1\, d\mu=0$.~\end{proof}

\subsection{Proof of Theorem~\ref{thm-LDGM}}
\label{subsec:main abstr}
  
As already mentioned in the paragraph after the statement of Theorem~\ref{thm-LDGM}, it suffices to show that~\eqref{eq:rightdyn} and~\eqref{eq:rightdynalpha2} hold. 
In the notation of Proposition~\ref{prop:dec},
$\E_\mu(e^{itv_n})=\int_{\Omega} R(t)^n 1\, d\mu=U(t, n)+
\Psi(t)^n$ and the LHS of equation~\eqref{eq:rightdyn}, for $\alpha\in (0,1)\cup (1,2)$, becomes

\begin{equation}
\label{eq:rightdyn2}
 \left|\int_{-\eps}^\eps (e^{-itN}-e^{-it(N+g(n))})
 \psi_Y(t)\frac{U(t,n)}{it}\, dt\right|=O\left(D(N)\right)+o\left(n \ell_0(N) N^{-\alpha}\right).
\end{equation}
where (in the notation of Theorem~\ref{thm-LDGM}),
\[
 D(N)=\begin{cases}
       \log N\,\ell(N)N^{-\alpha}, & \text{ if }\alpha\in (0,1),\\
       \,N^{-(\alpha-\delta)}, & \text{ if }\alpha\in (1,2].
      \end{cases}\]
      
Similarly, for $\alpha=2$,
equation~\eqref{eq:rightdynalpha2} becomes
      
      \begin{align}
\label{eq:rightdyn10}
 \left|\int_{-\eps}^\eps (e^{-itN}-e^{-it(N+g(n))})
 \psi_Y(t)\frac{U(t,n)}{it}\, dt\right|&=O(D(N))+ O\left(n \log N\, N^{-2}\right)+O\left(n \hat\ell(N) N^{-2}\right).
\end{align}
In what follows we show that~\eqref{eq:rightdyn2} and ~\eqref{eq:rightdyn10} hold by showing that for $\alpha\in (0,1)\cup (1,2)$,
\begin{equation}
\label{eq:rightdyn3}
 \left|\int_{-\eps}^\eps e^{-itN}
 \psi_Y(t)\frac{U(t,n)}{it}\, dt\right|=O\left(D(N)\right)+o\left(n \ell_0(N) N^{-\alpha}\right).
\end{equation}
 and for $\alpha=2$,
 \begin{equation}
\label{eq:rightdyn55555}
 \left|\int_{-\eps}^\eps e^{-itN}
 \psi_Y(t)\frac{U(t,n)}{it}\, dt\right|=O(D(N))+ O\left(n \log N\, N^{-2}\right)+O\left(n \hat\ell(N) N^{-2}\right).
\end{equation}

A similar argument shows that for $\alpha\in (0,1)\cup (1,2)$,
\begin{equation*}
 \left|\int_{-\eps}^\eps e^{-it(N+g(n))}
 \psi_Y(t)\frac{U(t,n)}{it}\, dt\right|=O\left(D(N+g(n))\right)
 +o\left(\frac{n \ell_0(N+g(n))}{ (N+g(n))^{\alpha}}\right)
 +O\left(\frac{n \log (N+g(n))}{(N+g(n))^{2}}\right),
\end{equation*}
 and a corresponding statement for $\alpha=2$,
concluding the proof of Theorem~\ref{thm-LDGM}.

The validity of equations~\eqref{eq:rightdyn3} and~\eqref{eq:rightdyn55555} will be shown via a `modulus of continuity' argument. To carry out such an argument, it is crucial that we have enough decay in $t$ in the expression of $U(t,n)$ as to
counteract the effect of the division by $t$ in the integrand of~\eqref{eq:rightdyn3}. Proposition~\ref{prop:dec}
tracks this type of decay in $t$ (along with finer continuity properties in $h$)
and it is in this sense that we call $U(t,n)$  a `good quantity'.

Before embarking on the final calculations, we need to recall one more technical lemma too.

\begin{lemma}{~\cite[Lemma 2.3]{MTejp}}
\label{lemma:20MT} 
Let $L:(0,\infty)\to (0,\infty)$ be a slowly varying function.
For all $\beta>0$, there exists $C>0$ so that for all $n\ge 1$,
\[
 \int_{0}^\eps t^{\beta}L(1/t)\, |\Psi(t)^n|\,dt\le C\frac{L(a_n)}{a_n^{1+\beta}}.
\]
\end{lemma}

We can now proceed to

\begin{pfof}{Equation~\eqref{eq:rightdyn3}}

{\bf{The case $\alpha\in(0,1)$}.} Set $f(t,n)=\psi_Y(t)\frac{U(t,n)}{t}$
and $K(n,N)=\int_{-\eps}^\eps e^{-itN}
 f(t,n)\, dt$.
To be able to apply the modulus of continuity argument we need to 
split the domain of integration (as to be able to deal with  the difference
coming from $\frac{1}{t}-\frac{1}{t-\frac{\pi}{N}}$). For this purpose, write
\begin{align*}
 K(n,N)&=\int_{-\frac{2\pi}{N}}^{\frac{2\pi}{N}} e^{-itN} f(t, n)\, dt
 +\int_{-\eps}^{-\frac{2\pi}{N}} e^{-itN} f(t, n)\, dt
 +\int_{\frac{2\pi}{N}}^\eps e^{-itN} f(t, n)\, dt\\
 &= K_0(n,N)+ K^-(n,N)+K^+(n,N).
\end{align*}

By Proposition~\ref{prop:dec} (i), $|f(t, n)|\ll \left(n |t|^{2\alpha-1}\ell\Big(\frac{1}{|t|}\Big)^2
   +|t|^{\alpha-1}\ell\Big(\frac{1}{|t|}\Big)\right)|\Psi(t)^n|$.
Thus, we compute (using Karamata's Theorem) that
\begin{align*}
 | K_0(n,N)|&\ll
 \int_{0}^{\frac{2\pi}{N}} t^{\alpha-1}\ell(1/t)\,dt +n\int_{0}^{\frac{2\pi}{N}} t^{2\alpha-1}\ell(1/t)^2\,dt\\
 & \ll  N^{-\alpha}\ell(N)+ n\,\ell(N)^2N^{-2\alpha}
  =o(n N^{-\alpha}\ell(N)).
\end{align*}

We estimate $K^+$ (via modulus of continuity), the estimate for $K^{-}$ is similar.
With the change of variable $t\to t-\pi/N$,
\[
K^+(n,N)=-\int_{(2\pi+\pi)/N}^{\eps+\pi/N} e^{-itN}f\left(t-\frac{\pi}{N}, n\right)\, dt.
\]
Summing the two expressions of $K^+$,
\begin{align*}
 2 K^+(n,N)& = -\int_{\eps}^{\eps+\pi/N}e^{-itN} f\left(t-\frac{\pi}{N}, n\right)\, dt
  +\int_{\frac{2\pi}{N}}^{ \frac{2\pi+\pi}{N} } e^{-itN} f\left(t-\frac{\pi}{N}, n\right)\, dt\\
 &\quad +\int_{\frac{2\pi}{N}}^{\eps} e^{-itN} \left(f(t, n)-f\left(t-\frac{\pi}{N}, n\right)\right)\, dt
 =K_1+K_2+K_3.
\end{align*}

Since the integrands of $K_1$ and $K_2$ are bounded, $|K_1|, |K_2|=o(n N^{-\alpha}\ell(N))$.
We continue with $K_3$, computing that
\begin{align*}
 |K_3|\le\ & \left|\int_{\frac{2\pi}{N}}^{\eps} e^{-itN}\left(\psi_Y(t)-\psi_Y\left(t-\frac{\pi}{N}\right)\right)\, \frac{U(t, n)}{t}\, dt\right|\\
 &+\left|\int_{\frac{2\pi}{N}}^{\eps} e^{-itN}\psi_Y\left(t-\frac{\pi}{N}\right)\,\frac{U(t, n)-U(t-\frac{\pi}{N}, n)}{t}\, dt\right|\\
  &+\left|\int_{\frac{2\pi}{N}}^{\eps} e^{-itN}\psi_Y\left(t-\frac{\pi}{N}\right)\,U\left(t-\frac{\pi}{n}, n\right)
 \left(\frac{1}{t}-\frac{1}{t-\frac{\pi}{N}}\right)\, dt\right|
 = K_3^1+K_3^2+K_3^3.
\end{align*}

Since $\psi_Y$ is Lipschitz and $|f(t, n)|\ll \left(n |t|^{2\alpha-1}\ell\Big(\frac{1}{|t|}\Big)^2
   +|t|^{\alpha-1}\ell\Big(\frac{1}{|t|}\Big)\right)\,|\Psi(t)^n|$,
\begin{align*}
 K_3^1 & \ll \frac{1}{N}\int_{\frac{2\pi}{N}}^{\eps} t^{\alpha-1}\ell(1/t)\, dt
 + \frac{n}{N}\int_{\frac{2\pi}{N}}^{\eps} t^{2\alpha-1}\ell(1/t)^2\, dt\\
 &\ll \frac{\ell(N)}{N^\alpha}+\frac{n \ell(N)}{N^{2\alpha}} =o(n N^{-\alpha}\ell(N)).
\end{align*}

So far all the terms are as good as desired. The loss in the statement
comes from treating the remaining integrals $K_3^2$ and $K_3^3$.
By Proposition~\ref{prop:dec} (i),
\[
 |U(t+h,n)-U(t, n)|\ll \left(|h|^{\alpha}\ell\Big(\frac{1}{|h|}\Big)
 +\, n|h|^{\alpha}\ell\Big(\frac{1}{|h|}\Big)\,|t|^{\alpha}\,\ell\Big(\frac{1}{|t|}\Big)\right)\,|\Psi(t)^n|.
\]
Taking $h=\pi/N$,
\begin{align}\label{eqk32}
\nonumber |K_3^2|&\ll \frac{\ell(N)}{N^\alpha}\int_{\frac{2\pi}{N}}^{\eps}\frac{1}{t}\, dt
+\frac{n\ell(N)}{N^\alpha}\int_{\frac{2\pi}{N}}^{\eps} t^{\alpha-1}\ell(1/t)|\Psi(t)^n|\, dt\\
&\ll \frac{\ell(N)\log N}{N^\alpha}+\frac{\ell(N)}{N^\alpha}\, n\,\int_{\frac{2\pi}{N}}^{\eps} t^{\alpha-1}\ell(1/t)|\Psi(t)^n|\, dt.
\end{align}
The second term is unproblematic. By Lemma~\ref{lemma:20MT}, with $\beta=\alpha-1$ and $L=\ell$,
\begin{align*}
 n\int_{\frac{2\pi}{N}}^\eps t^{\alpha-1}\ell(1/t)\,  |\Psi(t)^n|\,dt
 &\ll n\,a_n^{-\alpha}\ell(a_n)\ll 1.
\end{align*}
Hence, $K_3^2 \ll \frac{\ell(N)\log N}{N^\alpha}$.

For $K_3^3$, we use again that $|U(t, n)/t|\ll n |t|^{2\alpha-1}\ell\Big(\frac{1}{|t|}\Big)^2
   +|t|^{\alpha-1}\ell\Big(\frac{1}{|t|}\Big)$
and that $\left|\frac{1}{t}-\frac{1}{t-\frac{\pi}{N}}\right|\ll \frac{1}{N}\frac{1}{t^2}$.
So, by Karamata's Theorem,
\begin{align*}
 |K_3^3|&\ll \frac{1}{N}\int_{\frac{2\pi}{N}}^{\eps} t^{\alpha-2}\ell(1/t)\, dt
 +\frac{n}{N}\int_{\frac{2\pi}{N}}^{\eps} t^{2\alpha-2}\ell(1/t)^2\, dt\\
 &\ll \frac{1}{N^{\alpha}} \ell(N) + \frac{n\ell(N)^2}{N^{2\alpha}}
 \ll N^{-\alpha}\ell(N)+o(n N^{-\alpha}\ell(N)),
\end{align*}
ending the argument for case $\alpha\in(0,1)$.\\

{\bf{The case $\alpha\in(1,2]$}.} Set $M(n,N)=\int_{-\eps}^\eps e^{-itN}
 \psi_Y(t)\frac{U(t,n)}{it}\, dt$. 
 
Since the boundary terms cancel, integration by parts gives
\begin{align}\label{eq:mn}
\nonumber  M(n,N)&=\frac{1}{iN}\int_{-\eps}^\eps e^{-itN}
 \psi_Y'(t)\frac{U(t,n)}{it}\, dt
 +\frac{1}{iN}\int_{-\eps}^\eps e^{-itN}
 \psi_Y(t)\frac{U'(t,n)}{it}\, dt\\
\nonumber  & \quad  -\frac{1}{iN}\int_{-\eps}^\eps e^{-itN}
 \psi_Y(t)\frac{U(t,n)}{t^2}\, dt \\
 &=\frac{1}{iN} M_1(n,N)+\frac{1}{iN}M_2(n,N)
 +\frac{1}{iN}M_3(n,N).
\end{align}

Throughout the rest of the proof we take $\alpha'$ close to $\alpha$ (that is,
$\alpha-\alpha'$ is positive but as small as we want).
By Proposition~\ref{prop:dec} (ii),
$|U(t,n)|\ll n |t|^{2}+ |t|^{\alpha'}$. By Proposition~\ref{prop:dec} (iii),
$|U'(t,n)|\ll  \left(n\,|t|
     + |t|^{\alpha-1}\ell_0\left(\frac{1}{|t|}\right)+\,n^2\,|t|^{\alpha+1}\ell_0\left(\frac{1}{|t|}\right)\right)\,|\Psi(t)^n|$.
Hence, integrating by parts once more,
\begin{align}\label{eq:m1n}
 \nonumber \left|\frac{1}{iN} M_1(n,N)\right| &\ll \frac{1}{N^2}\int_{0}^\eps|\psi_Y''(t)|\frac{\left|U(t,n)\right|}{t}\, dt
 +\frac{1}{N^2}\int_{0}^\eps |\psi_Y'(t)|\frac{\left|U'(t,n)\right|}{t}\, dt\\
 &\quad +\frac{1}{N^2}\int_{0}^\eps |\psi_Y'(t)|\frac{\left|U(t,n)\right|}{t^2}\, dt\ll\frac{n}{N^2}= o\left(\frac{n \ell_0(N) }{N^{\alpha}}\right),
\end{align}
where we used that if $\alpha = 2$, we assumed infinite variance, so $\ell_0(N) \to \infty$. Here we have also used that $\psi_Y$ is $C^2$ (with bounded first and second derivative).

Next we estimate $M_2, M_3$ via the modulus of continuity argument 
similarly to the proof in the case $\alpha\in(0,1)$ above.

Regarding $M_2$,
\begin{align}\label{eq:m2222}
 \nonumber M_2(n,N)&=\int_{-\frac{2\pi}{N}}^{\frac{2\pi}{N}} e^{-itN} \psi_Y(t)\frac{U'(t,n)}{t}\, dt
 +\int_{-\eps}^{-\frac{2\pi}{N}} e^{-itN}\psi_Y(t) \frac{U'(t,n)}{t}\, dt\\
 &+\int_{\frac{2\pi}{N}}^\eps e^{-itN} \psi_Y(t)\frac{U'(t,n)}{t}\, dt
 = M_2^0(n,N)+M_2^-(n,N)+M_2^+(n,N).
\end{align}
Next, 
\begin{align*}
 M_2^0(n,N)&\ll n\int_{0}^{\frac{2\pi}{N}}1\, dt
 +\int_{0}^{\frac{2\pi}{N}}t^{\alpha-2}\ell_0\left(\frac{1}{t}\right)\, dt
  +n^2\int_{0}^{\frac{2\pi}{N}}t^{\alpha}\ell_0\left(\frac{1}{t}\right)\, dt\\
 &\ll\frac{n }{N}+ 
 \frac{\ell_0(N) }{N^{\alpha-1}} +\frac{n^2\ell_0(N)}{N^{\alpha+1}}
 =O\left(\frac{n }{N}\right)+o\left(\frac{n \ell_0(N) }{N^{\alpha-1}}\right).
\end{align*}

Similarly to the argument used above for the case $\alpha\in(0,1)$ (in estimating $K^+, K^-$),
it suffices to estimate $M_2^+(n,N)$.
Set $g(t, n)=\psi_Y(t)\frac{U'(t,n)}{t}$ and compute that

\begin{align}\label{eq:m20n}
 \nonumber 2 M_2^+&(n,N) = -\int_{\eps}^{\eps+\frac{\pi}{N}}e^{-itN} g\left(t-\frac{\pi}{N}, n\right)\, dt
  +\int_{\frac{2\pi}{N}}^{\frac{2\pi+\pi}{N} } e^{-itN} g\left(t-\frac{\pi}{N}, n\right)\, dt\\
 &+\int_{\frac{2\pi}{N}}^{\eps} e^{-itN}
 \left(g(t, n)-g\left(t-\frac{\pi}{N}, n\right)\right)\, dt
 =M_2^1+M_2^2+M_2^3.
\end{align}
The estimates for $M_2^1$ and $M_2^2$ go similarly to the estimate for $M_2^0(n,N)$ above
and give $|M_2^1|=O(\frac1N) = o\left(\frac{n \ell_0(N) }{N^{\alpha-1}}\right)$
and $|M_2^2|=o\left(\frac{n \ell_0(N) }{N^{\alpha-1}}\right)$.

Regarding $M_2^3$, recalling that $\psi_Y$ is bounded,  we compute that
\begin{align*}
 |M_2^3|& \ll \int_{\frac{2\pi}{N}}^{\eps}\left|\psi_Y(t)-\psi_Y\left(t-\frac{\pi}{N}\right)\right|
 \frac{|U'(t,n)|}{t}\, dt
  +\int_{\frac{2\pi}{N}}^{\eps}\psi_Y\left(t-\frac{\pi}{N}\right)\frac{\left| U'(t,n)-
  U'\left(t-\frac{\pi}{N}, n\right)\right|}{t}\, dt\\
 &\quad +
 \int_{\frac{2\pi}{N}}^{\eps} \psi_Y\left(t-\frac{\pi}{N}\right)\left|U'\left(t-\frac{\pi}{N}, n\right)\right|\,\left|\frac{1}{t}-\frac{1}{t-\frac{\pi}{N}}\right|\, dt= L_1+L_2+L_3.
\end{align*}
Since $\psi_Y$ is Lipschitz, recalling that $|U'(t,n)/t|\ll  \left(n
     + |t|^{\alpha-2}\ell_0\left(\frac{1}{|t|}\right)+\,n^2\,|t|^{\alpha}\ell_0\left(\frac{1}{|t|}\right)\right)\,|\Psi(t)^n|$ (and also using Lemma~\ref{lemma:20MT}, with $\beta\in\{\alpha-2,\alpha\}$, and $L=\ell$), we get
\begin{align}\label{eql1}
 L_1\ll\frac{n}{N}\int_{\frac{2\pi}{N}}^{\eps} |\Psi(t)^n|\, dt \ll\frac{n}{a_n}\frac{1}{N},
\end{align}
where we have used that $a_n^\alpha\sim n\ell_0(a_n)$.

By Proposition~\ref{prop:dec} (iii), 
 \begin{align*}
 \Big|U'(t+h,n)- & U'(t, n)\Big|\ll
 |h|^{\alpha-1}\ell_0\left(\frac{1}{|h|}\right)\,|\Psi(t)^n|\\
 &+n\left(|t|\,|h|^{\alpha'-1}\ell_0\left(\frac{1}{|h|}\right)
 +|h|
 \right)|\Psi(t)^n|\\
 &+n^2 \left(|t|^{2\alpha-1}|h|\ell_0\left(\frac{1}{|t|}\right)^2+|t|^2|h|^{\alpha-1}\ell_0\left(\frac{1}{|h|}\right)+ |h|\, |t|^\alpha\ell_0\left(\frac{1}{|t|}\right)\right)\,|\Psi(t)^n|\\
 &+n^3 |t|^{2\alpha}|h|\ell_0\left(\frac{1}{|t|}\right)^2\,|\Psi(t)^n|.
    \end{align*}
Using this estimate with $h = \pi/N$ and
recalling that $|\Psi(t)|\le e^{-c t^\alpha\ell_0(1/t)}$,
 \begin{align*}
  L_2\ll\ &\frac{\ell_0(N) }{N^{\alpha-1}}\int_{\frac{2\pi}{N}}^{\eps}\frac{1}{t} \,dt+\frac{n \ell_0(N) }{N^{\alpha'-1}}\int_{\frac{2\pi}{N}}^{\eps}|\Psi(t)^n|\, dt+\frac{n}{N}\int_{\frac{2\pi}{N}}^{\eps}\frac{1}{t}\, dt\\
  &+\frac{n^2 \ell_0(N) }{N^{\alpha-1}}\int_{\frac{2\pi}{N}}^{\eps} t^{2\alpha-2}\ell_0\left(\frac{1}{t}\right)\,|\Psi(t)^n|\,dt
  +\frac{n^2 \ell_0(N) }{N^{\alpha-1}}\int_{\frac{2\pi}{N}}^{\eps} t\,|\Psi(t)^n|\,dt\\
  &+\frac{n^2 }{N}\int_{\frac{2\pi}{N}}^{\eps} t^{\alpha-1}\ell_0\left(\frac{1}{t}\right)\,|\Psi(t)^n|\,dt\\
  &+\frac{n^3}{N} \int_{\frac{2\pi}{N}}^{\eps}
 t^{2\alpha-1}\, \ell_0\left(\frac{1}{t}\right)^2\,|\Psi(t)^n|\, dt.
  \end{align*}
  By Lemma~\ref{lemma:20MT} with $\beta=\alpha-1$, and respectively $\beta=2\alpha-2$,  and $L=\ell_0$
  \begin{equation}
  \label{eq:mtuse}
    \int_{\frac{2\pi}{N}}^{\eps} t^{\alpha-1}\ell_0\left(\frac{1}{t}\right)\,|\Psi(t)^n|\,dt\ll
   \frac{\ell_0(a_n)}{a_n^\alpha}\ll\frac{1}{n},
   \end{equation}

  \begin{align*}
   \int_{\frac{2\pi}{N}}^{\eps} t^{2\alpha-2}\ell_0\left(\frac{1}{t}\right)\,|\Psi(t)^n|\,dt
   \ll \frac{\ell_0(a_n)}{a_n^{2\alpha-1}}\ll \frac{1}{n^{2-1/\alpha-\delta^*}},\,\text{ for any } \delta^*>0,
    \end{align*}
and using that $a_n^\alpha\sim n\ell_0(a_n)$,
  \begin{align}\label{eq:mt555555}
      \int_{\frac{2\pi}{N}}^{\eps}
 t^{2\alpha-1}\, \ell_0\left(\frac{1}{t}\right)^2\,|\Psi(t)^n|\, dt\ll
 \frac{\ell_0(a_n)^2}{a_n^{2\alpha}}\ll \frac{1}{n^{2}}.
  \end{align}
  Also, by Lemma~\ref{lemma:20MT} with $\beta=1$ and $L=1$,
  \begin{align}\label{eq:m111}
   \int_{\frac{2\pi}{N}}^{\eps} t\,|\Psi(t)^n|\, dt\ll \frac{1}{a_n^{2}}\ll \frac{1}{(n\ell_0(a_n))^{2/\alpha}}
   \ll \frac{1}{n^{2/\alpha-\delta^*}},\,\text{ for any } \delta^*>0.
  \end{align}

Using the same computation as in~\eqref{eql1},
$\int_{\frac{2\pi}{N}}^{\eps}|\Psi(t)^n|\, dt\ll a_n^{-1}$.
These estimates  together with~\eqref{eq:mtuse},~\eqref{eq:mt555555} and~\eqref{eq:m111} (after a change of variables $t \to \sigma/a_n$) imply that
  \begin{align*}
 L_2 &\ll \frac{n \ell_0(N) }{N^{\alpha'-1}}\frac{1}{a_n}+\frac{n\log N}{N}
  +\frac{\log N\, \ell_0(N)}{N^{\alpha-1}}\\
  &+\frac{n^2 \ell_0(N)}{N^{\alpha-1}}\frac{1}{n^{2-1/\alpha-\delta^*}}
  +\frac{n^2 }{N}\frac{1}{n}+ \frac{n^2 \ell_0(N) }{N^{\alpha-1}}\frac{1}{n^{2/\alpha-\delta^*}}+
  \frac{n^3 }{N}\frac{1}{n^{2}}\\
  &=O\left(\frac{\log N\,\ell_0(N)}{N^{\alpha-1}}\right)+O\left(\frac{n \log N}{N}\right)
  +o\left(\frac{n \ell_0(N)}{ N^{\alpha'-1}}\right).
 \end{align*}
Here we have used that $\frac{n^2}{n^{2-1/\alpha-\delta^*}}<n$, $\frac{n^2}{n^{2/\alpha-\delta^*}}<n$  and that $a_n^\alpha\sim n\ell_0(a_n)$.
 The argument for estimating $L_3$ goes similarly to the argument
 used for $K_3^2$ in~\eqref{eqk32} (inside the proof for the case $\alpha\in (0,1)$ above).
  Recall that $|U'(t,n)|\ll  \left(|t|^{\alpha-1}\ell_0\left(\frac{1}{|t|}\right)+ n |t|
  +n^2|t|^{\alpha+1}\ell_0\left(\frac{1}{|t|}\right)\right)\,|\Psi(t)^n|$
  and $\left|\frac{1}{t}-\frac{1}{t-\frac{\pi}{N}}\right|\ll \frac{1}{N}\frac{1}{t^2}$.
 Hence, using Karamata's Theorem,
 \begin{align*}
  L_3&\ll
  \frac{1}{N}\int_{\frac{2\pi}{N}}^{\eps}t^{\alpha-3}\ell_0\left(\frac{1}{t}\right)\, dt+\frac{n}{N}\int_{\frac{2\pi}{N}}^{\eps} \frac{1}{t}\, dt+\frac{n^2}{N}\int_{\frac{2\pi}{N}}^{\eps}t^{\alpha-1}\ell_0\left(\frac{1}{t}\right)\,|\Psi(t)^n|\, dt\\
  & \ll \frac{\ell_0(N)}{N^{\alpha-1}} 
  +\frac{n\,\log N}{N}
 + \frac{n}{N}.
 \end{align*}
 where we have used~\eqref{eq:mtuse}.
Putting all the above together and recalling~\eqref{eq:m20n},
\begin{align}\label{eq:m2n}
 \left|\frac{1}{iN} M_2(n,N)\right|=O\left(\frac{\log N\,\ell_0(N)}{N^{\alpha'}}\right)+O\left(\frac{n \log N}{N^2}\right)
 +o\left(\frac{n \ell_0(N)}{ N^{\alpha}}\right).
\end{align}

It remains to estimate $M_3(n,N)$ in~\eqref{eq:mn}.
We use the modulus of continuity argument again.
Similarly to~\eqref{eq:m2222},
\begin{align*}
 M_3(n,N)&=\int_{-\frac{2\pi}{N}}^{\frac{2\pi}{N}} e^{-itN} \psi_Y(t)\frac{U(t,n)}{t^2}\, dt
 +\int_{-\eps}^{-\frac{2\pi}{N}} e^{-itN}\psi_Y(t) \frac{U(t,n)}{t^2}\, dt\\
 &+\int_{\frac{2\pi}{N}}^\eps e^{-itN} \psi_Y(t)\frac{U(t,n)}{t^2}\, dt
 = M_3^0(n,N)+M_3^-(n,N)+M_3^+(n,N).
\end{align*}
Recall that by Proposition~\ref{prop:dec}(ii),
\(
|U(t, n)|\ll \left(|t|^{\alpha'}+\,n |t|^2\right)\,|\Psi(t)^n|,
    \)
    for any $\alpha'\in (1,\alpha)$.
An easy computation shows that $|M_3^0(n,N)|=O( N^{1-\alpha'})+ O(n N^{-1})=O(N^{1-\alpha'})+o\left(n \ell_0(N) N^{\alpha-1}\right)$,
where we note that if $\alpha = 2$, we assumed infinite variance, so $\ell_0(N) \to \infty$.

Similarly to the argument used for~\eqref{eq:m2222}, it suffices to estimate $M_3^+(n,N)$. Compute that
\begin{align*}
 |M_3^+|& \ll \int_{\frac{2\pi}{N}}^{\eps}\left|\psi_Y(t)-\psi_Y\left(t-\frac{\pi}{N}\right)\right|
 \frac{|U(t,n)|}{t^2}\, dt
  +\int_{\frac{2\pi}{N}}^{\eps}\psi_Y\left(t-\frac{\pi}{N}\right)\frac{\left|U(t,n)-
  U\left(t-\frac{\pi}{N}, n\right)\right|}{t^2}\, dt\\
 &\quad +
 \int_{\frac{2\pi}{N}}^{\eps} \psi_Y\left(t-\frac{\pi}{N}\right)\left|U\left(t-\frac{\pi}{N}, n\right)\right|\,\left|\frac{1}{t^2}-\frac{1}{\left(t-\frac{\pi}{N}\right)^2}\right|\, dt= S_1+S_2+S_3.
\end{align*}
Clearly, $S_1\ll n N^{-1}=o\left(\frac{n \ell_0(N)}{ N^{\alpha-1}}\right)$.

It remains to estimate $S_2$ and $S_3$.
By Proposition~\ref{prop:dec}(ii),  for any $\alpha'\in (1,\alpha)$,
  \[
   | U(t+h,n)-U(t, n)|\ll
  |h|\, |t|^{\alpha'-1}\,|\Psi(t)^n|+n\,|h|\,|t|.
  \]

Taking $h=\pi/N$,
  \begin{align*}
   S_2& \ll \frac{1}{N}\int_{\frac{2\pi}{N}}^\eps t^{\alpha'-3}\, dt
   +\frac{n}{N}\int_{\frac{2\pi}{N}}^{\eps} t^{-1}\, dt
   \ll \frac{1}{N^{\alpha'-1}}+\frac{n\log N}{N}.
  \end{align*}
  
  Finally, recall again that by Proposition~\ref{prop:dec} (ii),
$|U(t,n)|\ll |t|^{\alpha'} +n |t|^{2}$,  for any $\alpha'\in (1,\alpha)$. Hence,
\[
 S_3\ll\frac{1}{N}\int_{\frac{2\pi}{N}}^\eps t^{\alpha'-3}\, dt
 +\frac{n}{N}\int_{\frac{2\pi}{N}}^{\eps}t^{-1}\, dt
 \ll \frac{1}{N^{\alpha'-1}}+\frac{n\log N}{N}.
\]
Thus, 
\begin{align*}|M_3^+|=
 \begin{cases}
   O\left(\frac{1}{N^{\alpha'-1}}\right)+o\left(\frac{n \ell_0(N)}{ N^{\alpha-1}}\right)&\text{ if }\alpha\in (1,2),\\
   O\left(\frac{1}{N^{\alpha'-1}}\right)+O\left(\frac{n \left(\ell_0(N)+\log N\right)}{N}\right)&\text{ if }\alpha=2.
 \end{cases}
\end{align*}The integral $|M_3^-|$ can be estimated similarly and gives similar estimates.
Altogether,
  \begin{align*} \left|\frac{1}{iN}M_3(n, N)\right|=
 \begin{cases}
   O\left(\frac{1}{N^{\alpha'}}\right)+o\left(\frac{n \ell_0(N)}{ N^{\alpha}}\right)&\text{ if }\alpha\in (1,2),\\[2mm]
   O\left(\frac{1}{N^{\alpha'}}\right)+O\left(\frac{n(\ell_0(N)+ \log N)}{ N^2}\right)&\text{ if }\alpha=2.
 \end{cases}
\end{align*}

 The conclusion follows from the previous displayed equation together
with~\eqref{eq:m2n},~\eqref{eq:m1n} (and recalling~\eqref{eq:mn}), and also by recalling that $\log N\ell_0(N)\ll N^{\delta_0}$, for any $\delta_0$ arbitrarily small.~\end{pfof}

\section{Applications}
\label{sec:appl}

\subsection{Gibbs-Markov maps}
\label{sec-GM}

Roughly speaking, Gibbs-Markov maps are infinite branch uniformly expanding maps with bounded distortion and big images. We recall the definitions in more detail.

Let $(\Omega,\mu)$ be a probability space with an at most countable measurable partition $\{\Omega_j\}$, and let $T:\Omega\to \Omega$ be an ergodic measure-preserving transformation.
Define $s(y,y')$ to be the least integer $n\ge0$ such that $T^ny$ and $T^ny'$ lie in distinct partition elements.
Assuming  that $s(y,y')=\infty$ if and only if $y=y'$ one obtains that $d_\theta(y,y')=\theta^{s(z,z')}$ is a metric
for $\theta\in(0,1)$, 

Let $g=\frac{d\mu}{d\mu\circ T}:\Omega\to\R$.
We say that $T$ is a Gibbs-Markov map if
\begin{itemize}

\parskip = -2pt
\item $T (\Omega_j)$ is a union of partition elements and $T|_{\Omega_j}:\Omega_j\to T(\Omega_j)$ is a measurable bijection for each $j\ge1$;
\item $\inf_j\mu(T(\Omega_j))>0$;
\item
There are constants $C>0$, $\theta\in(0,1)$ such that
$|\log g(y)-\log g(y')|\le Cd_\theta(y,y')$ for all $y,y'\in \Omega_j$, $j\ge1$.
\end{itemize}
For background on Gibbs-Markov maps see, for instance,~\cite[Chapter 4]{Aaronson} and~\cite{AD01}.

Given $f:\Omega\to\R$, let
\[
D_jf=\sup_{y,y'\in\Omega_j,\,y\neq y'}|f(y)-f(y')|/d_\theta(y,y'),\qquad |f|_\theta=\sup_{j\ge1}D_j f.
\]
We let the Banach space $\cB_\theta\subset L^\infty$  consist of functions
$f:\Omega\to\R$ such that $|f|_\theta<\infty$ with norm
$\|f\|_\theta=|f|_\infty+|f|_\theta<\infty$.

\begin{prop} \label{prop-GM}
Assume $T$ is a mixing Gibbs-Markov map and
let $v:\Omega\to\R$ with $\int_\Omega v^2\,d\mu=\infty$ and $|v|_\theta<\infty$. Assume moreover that
 there exists $C>0$ so that
\begin{equation}
 \label{eq: extrass}
 \left|v|_{\Omega_j}\right|_\theta\le C\inf_j(|v|_{\Omega_j}|),\text{ for all }\Omega_j.
\end{equation}

Fix $\alpha\in(0,1)\cup(1,2]$ and 
assume that the tails of $v$ satisfy (H1).
Then conditions (H1)--(H3) are satisfied with Banach space $\cB=\cB_\theta$.
\end{prop}

\begin{proof}
 Conditions (H1) is an assumption on $v$ and condition (H3) is well known 
 (see \cite[Chapter 4]{Aaronson} and~\cite{AD01}).
 Conditions (H2)(i), (ii) are again well known and condition (H2)(iv)
 has been verified in~\cite[Proof of Proposition 3.10]{MTejp}. 
 
 It remains to verify (H2)(ii). First note that
 \[
  \|R(t+h)-R(t)-ih R'(0)\|\le \left\|R\left(e^{itv}\left(e^{ihv}-1-ihv\right)\right)\right\|
  +\left\|R\left(ihv(e^{itv}-1)\right)\right\|.
 \]
 Let $\phi\in\{v(e^{itv}-1), e^{itv}\left(e^{ihv}-1-ihv\right)\}$ and recall that $\phi$ satisfies assumption~\eqref{eq: extrass}.
 As shown in~\cite[Proof of Proposition 12.1]{MelTer17}, under~\eqref{eq: extrass},
 there exists $C>0$ so that for all $f\in\cB_\theta$,
 \[
  \|R(\phi\, f)\|\le C\|f\|\,\|\phi\|_{L^1(\mu)}.
 \]
Since we know $v$ satisfies (H1), we have for any $\alpha'\in (1,\alpha)$,
\begin{align*}
 \|v(e^{itv}-1)\|_{L^1(\mu)}\ll |t|^{\alpha'-1}\text{ and }
 \left\|e^{ihv}-1-ihv\right\|_{L^1(\mu)}\ll |h|^{\alpha'}.
\end{align*}
Thus, $\left\|R(e^{itv}\left(e^{ihv}-1-ihv\right)\right\|\ll |h|^{\alpha'}$
and $\left\|R\left(ihv(e^{itv}-1)\right)\right\|\ll |h|\, |t|^{\alpha'-1}$.~\end{proof}

\subsection{Non-Markov, expanding interval  maps (aka AFU maps)}

Let $\Omega=[0,1]$ with measurable partition $\{I\}$ consisting of open intervals.  A map
$T:\Omega\to \Omega$ is called AFU if $T|_I$ is $C^2$ and strictly monotone for each $I$, and
\begin{itemize}

\parskip = -2pt
\item[(A)] (Adler's condition)   $T''/(T')^2$ is bounded on $\bigcup I$.
\item[(F)] (finite images)  The set of images $\{ T I\}$ is finite.
\item[(U)] (uniform expansion)  There exists $\rho>1$ such that
$|T'|\ge\rho$ on $\bigcup I$.
\end{itemize}
Such maps have been introduced in dynamics in~\cite{Zweimuller98} (see also~\cite{ADSZ04}).  Since AFU maps are not necessarily Markov, H\"older spaces are not preserved by the transfer operator of $T$ and it is standard to consider the space of bounded variation functions.
Accordingly, we define the Banach space $\cB = BV\subset L^\infty$
to consist of functions $f:\Omega\to\R$ such that
$\Var(f)<\infty$
with norm $\|f\|=|f|_\infty+\Var(f)$, where
\[
\Var(f) = \inf_{g \sim f} \sup_{0 = y_0 < ... < y_k = 1} \sum_{i=1}^k |g(y_i)-g(y_{i-1})|,
\]
where $g \sim f$ if $f$ and $g$ differ on a null set, denotes the variation of the (equivalence class) of $f$.
Also, we let $\Var_I (f)$ denote the variation of $f$ on $I$.

We suppose that $T:\Omega\to \Omega$ is topologically mixing.
Then there is a unique absolutely continuous $T$-invariant probability measure $\mu$, and $\mu$ is mixing.

\begin{prop} \label{prop-AFU}
Assume $T$ is a topologically mixing AFU map and
let $v:\Omega\to\R$ with $\int_\Omega v^2\,d\mu=\infty$ and $\supI\Var_I (v)<\infty$.
Fix $\alpha\in(0,1)\cup(1,2]$ and 
assume that the tails of $v$ satisfy (H1).  

Then conditions (H1)--(H3) are satisfied with Banach space $\cB=\BV$.
\end{prop}

\begin{proof}
Condition (H1) is an assumption and condition~(H3) is well-known for mixing AFU maps (see~\cite{Zweimuller98}).
Apart from (H2)(iii), the remaining items of (H2) have been clarified in~\cite[Proof of Proposition 3.11]{MTejp}.

The verification of (H2)(iii) is similar to the one in Proposition~\ref{prop-GM} with the only change that assumption~\eqref{eq: extrass} 
is replaced by 
$\supI\Var_I (v)<\infty$. The assumption
$\supI\Var_I (v)<\infty$ ensures that $\Var_I (e^{itv})\ll |t|\Var_I (v)\ll |t|$ and similarly, $\Var_I (e^{ihv})\ll |h|$.
Using this and proceeding as in~\cite[Proof of Proposition 3.11]{MTejp}
we obtain that
for $\phi=v(e^{itv}-1)$,
we have that
\(
  \|R(\phi\, f)\|\le C\|f\|\,\|\phi\|_{L^1(\mu)}.
 \)
 
 A similar calculation which we sketch below gives the same estimate
 for $\phi= e^{itv}\left(e^{ihv}-1-ihv\right)$.
 Compute that $\|R(e^{itv}\left(e^{ihv}-1-ihv\right)\,f)\|= S_1+S_2+S_3$,
 where
 \begin{align*}
S_1 & =\sum_I \mu(I)\sup_I|e^{ihv}-1-ihv|\sup_I |f| \ll \|e^{ihv}-1-ihv\|_{L^1(\mu)},\\
S_2 & =\sum_I \mu(I)\sup_I|e^{itv}|\left(\Var_I (e^{ihv}) +|h| \Var_I(v)\right)\sup_I |f|\ll |h|,\\
S_3&=\sum_I \mu(I)\sup_I|e^{itv}\left(e^{ihv}-1-ihv\right)|\Var_I (f) \ll \|e^{ihv}-1-ihv\|_{L^1(\mu)}.
\end{align*}

From here onward the proof goes exactly the same as in Proposition~\ref{prop-GM}.~\end{proof}\\

\textbf{Acknowledgement.} We are grateful to the referees for their detailed reading and for pointing out a serious number of very good suggestions and
necessary corrections.

\end{document}